\documentclass{amsart}
\usepackage{MnSymbol}
\usepackage[mathscr]{euscript}
\usepackage{scalerel}
\usepackage{colonequals}
\usepackage{csquotes}
\usepackage[normalem]{ulem}

\usepackage{comment}

\usepackage{hyperref}
\hypersetup{
    colorlinks,
    citecolor=black,
    filecolor=black,
    linkcolor=black,
    urlcolor=black
}

\DeclareFontFamily{OT1}{pzc}{}
\DeclareFontShape{OT1}{pzc}{m}{it}{<-> s * [1.100] pzcmi7t}{}
\DeclareMathAlphabet{\mathpzc}{OT1}{pzc}{m}{it}

    \usepackage{etoolbox}
    \patchcmd{\section}{\scshape}{\large\bfseries}{}{}
    \makeatletter
    \renewcommand{\@secnumfont}{\bfseries}
    \makeatother

\usepackage{tikz-cd}
\tikzcdset{arrow style=math font}
\tikzcdset{scale cd/.style={every label/.append style={scale=#1},
    cells={nodes={scale=#1}}}}
\usepackage{tikz}
\usetikzlibrary{arrows}
\usepackage{caption}
\usepackage{float}

\numberwithin{equation}{section}
\newtheorem{theorem}{Theorem}[section]
\newtheorem*{theorem*}{Theorem}
\newtheorem{corollary}[theorem]{Corollary}
\newtheorem{lemma}[theorem]{Lemma}
\newtheorem{proposition}[theorem]{Proposition}
\theoremstyle{definition}

\newtheorem*{question*}{Question}
\newtheorem*{conjecture*}{Conjecture}

\newtheorem{definition}[theorem]{Definition}
\newtheorem{remark}[theorem]{Remark}
\newtheorem{example}[theorem]{Example}

\def\mono{\rightarrowtail}
\def\epi{\twoheadrightarrow}
\def\RR{\mathbb{R}}
\def\ZZ{\mathbb{Z}}
\def\VR{\mathscr{VR}}
\def\vr{\scalebox{0.8}[1.0]{\tt VR}}
\def\sk{\mathsf{sk}}
\def\CC{\mathcal{C}}

\def\KK{\mathbb{K}}

\def\MH{\mathsf{MH}}
\def\MH{\mathsf{MH}}
\def\UU{\mathcal{U}}
\def\VV{\mathcal{V}}
\DeclareMathOperator{\colim}{\mathrm{colim}}

\def\DD{\mathcal{D}}
\def\FF{\mathcal{F}} 
\def\Ch{\text{\rm \v{C}}} 
\def\SS{\mathsf{SS}}
\def\SC{\mathsf{SC}}
\def\sSet{\mathsf{sSet}}
\def\SCpx{\mathsf{SCpx}}

\def\Mat{\mathsf{Mat}}
\def\Ell{\mathsf{Ell}}
\def\LH{\mathsf{LH}}
\def\diam{\mathsf{diam}}

\DeclareMathOperator{\diag}{\mathsf{diag}}
\DeclareMathOperator{\hocolim}{\mathrm{hocolim}}
\def\WW{\mathpzc{w}}
\def\ww{\mathtt{w}}
\def\PMet{\mathsf{PMet}}
\def\Met{\mathsf{Met}}
\def\Kol{\mathsf{Kol}}
\def\Set{\mathsf{Set}}
\def\Pro{\mathsf{Pro}}
\def\GG{\mathcal{G}}
\def\Hom{\mathsf{Hom}} 
\def\HH{\mathcal{H}}
\def\Dist{\mathsf{Dist}}
\def\Fun{\mathsf{Fun}}

\let\oldtocsection=\tocsection 
\let\oldtocsubsection=\tocsubsection 
\renewcommand{\tocsection}[2]{\hspace{0mm}\oldtocsection{#1}{#2}}
\renewcommand{\tocsubsection}[2]{\hspace{1em}\oldtocsubsection{#1}{#2}}

\title{On $\ell_p$-Vietoris-Rips complexes}

\author{Sergei O. Ivanov} 
\address{
Beijing Institute of Mathematical Sciences and Applications (BIMSA)}
\email{ivanov.s.o.1986@gmail.com, ivanov.s.o.1986@bimsa.cn}

\author{Xiaomeng Xu} 
\email{xiaomeng.x.xu@gmail.com}

\begin{document}

\begin{abstract} We study the concepts of the $\ell_p$-Vietoris-Rips simplicial set and the $\ell_p$-Vietoris-Rips complex of a metric space, where $1\leq p \leq \infty.$ This theory unifies two established theories: for $p=\infty,$ this is the classical theory of Vietoris-Rips complexes, and for $p=1,$ this corresponds to the blurred magnitude homology theory. We prove several results that are known for the Vietoris-Rips complex in the general case: (1) we prove a stability theorem for the corresponding version of the   persistent homology; (2) we show that, for a compact Riemannian manifold and a sufficiently small scale parameter, all the ``$\ell_p$-Vietoris-Rips spaces'' are homotopy equivalent to the manifold; (3) we demonstrate that the $\ell_p$-Vietoris-Rips spaces are invariant (up to homotopy) under taking the metric completion. Additionally, we show that the limit of the homology groups of the $\ell_p$-Vietoris-Rips spaces, as the scale parameter tends to zero, does not depend on $p$; and that the homology groups of the $\ell_p$-Vietoris-Rips spaces commute with filtered colimits of metric spaces.  
\end{abstract}

\maketitle

\setcounter{tocdepth}{1}
\tableofcontents

\section{Introduction}

The Vietoris-Rips complex of a metric space was introduced by Vietoris to define the homology theory of a metric space \cite{vietoris1927hoheren}. Later, Rips used this complex to study hyperbolic groups. The work of Rips was popularized by Gromov in  \cite{gromov1987hyperbolic}. Hausmann proved that for a compact 
Riemannian manifold $M$ and a sufficiently small scale parameter $r,$ the geometric realization of the Vietoris-Rips complex $\vr_{<r} M$ is homotopy equivalent to $M$ \cite{hausmann1994vietoris}. He also showed that the Vietoris-Rips complex is invariant up to homotopy under taking the metric completion $|\vr_{<r} X| \simeq |\vr_{<r}\hat X|.$ 
Adamaszek and Adams identified the homotopy type of the Vietoris-Rips complex of the circle for any scale parameter  $r$  \cite{adamaszek2017vietoris}. They proved that all odd spheres $S^1,S^3,S^5,\dots$ appear as homotopy types of the Vietoris-Rips complex $\vr_{<r} S^1$ for some $r.$ 
Carlsson-Carlsson-De Silva applied the Vietoris-Rips complex in topological data analysis \cite{carlsson2006algebraic}, and Chazal with his coauthors proved the stability theorem for the Vietoris-Rips complex 
\cite{chazal2009proximity} (see also \cite{chazal2014persistence}).

The magnitude function of a  compact metric space was introduced by  Leinster \cite{leinster2013magnitude}. It was treated as an analog of the Euler characteristic of a topological space. The corresponding (co)homology theory was developed by Hepworth and Willerton for the case of graphs \cite{hepworth2017categorifying}, and extended to all metric spaces by Leinster, Shulman \cite{leinster2021magnitude} and Hepworth  \cite{hepworth2022magnitude}   under the name of magnitude (co)homology theory. Otter introduced a persistent  version of the magnitude homology under the name of blurred magnitude homology \cite{otter2018magnitude}. She proved that the homology groups of the Vietoris-Rips complex  and the blurred magnitude homology  have the same limit, as the scale parameter tends to zero. 
 Govc and Hepworth \cite{govc2021persistent} showed that the magnitude function of a  finite metric space can be also recovered from the blurred magnitude homology. 
 
In an unpublished preprint \cite{cho2019quantales}, Cho used the language of enriched categories over quantales to develop a theory that includes the Vietoris-Rips complex and the blurred homology theory as  special cases. For each $1\leq p\leq \infty,$ a metric space $X$ and a real number $r>0,$ he defines a simplicial set which we call the \emph{$\ell_p$-Vietoris-Rips simplicial set}. 
For $p=\infty,$ the geometric realization of the simplicial set is homotopy equivalent to the geometric realization of the  classical Vietoris-Rips complex, and for $p=1,$ its homology is the blurred magnitude homology.

We provide an alternative definition of the $\ell_p$-Vietoris-Rips simplicial sets in terms of the \emph{$\ell_p$-weight} of a tuple of points in a metric space. Our definition is shown to be equivalent to the definition of Cho (Proposition \ref{prop:Cho}). 
Since simplicial complexes are more convenient for data analysis, we also define $\ell_p$-Vietoris-Rips complexes in terms of $\ell_p$-weights of finite subsets. We develop a theory that works in both of these two settings 
---  the setting of simplicial sets and the setting of simplicial complexes. 
In particular, we prove the stability theorem, the result about Riemannian manifolds, the result about metric completions, the result about limit homology and commutativity with filtered colimits for both of these two settings. 
More precisely, we introduce a general framework that includes both of these two settings and prove the results in this general framework. 

The general framework is based on a general concept we introduce called a \emph{distance matrix norm}. Roughly speaking, a distance matrix norm $\nu$ is a norm defined on the ``space'' of all distance matrices of all tuples. A distance matrix norm $\nu$ defines a $\nu$-weight of a tuple of points in a metric space, which is used to define the $\nu$-Vietoris-Rips simplicial set. 
We show that the theory of both $\ell_p$-Vietoris-Rips simplicial sets and $\ell_p$-Vietoris-Rips complexes can be reduced to the theory of $\nu$-Vietoris-Rips simplicial sets for appropriate distance matrix norms $\nu$, and prove all relevant statements in the general context of any distance matrix norm. 

\subsection{Definitions of the $\ell_p$-Vietoris-Rips spaces} Let us define the main characters of the paper: the $\ell_p$-Vietoris-Rips simplicial sets $\VR^p_{<r}X$ and $\VR^p_{\leq r}X$ and the $\ell_p$-Vietoris-Rips complexes $\vr^p_{<r} X$ and $\vr^p_{\leq r} X$ of a metric space $X$. 
The definition is based on the concept of the $\ell_p$-weight of a tuple of points in a metric space. It is defined by the formula
\begin{equation}
   \WW_p(x_0,\dots,x_n) = \underset{0\leq i_0< \dots < i_m\leq n }\max \| d(x_{i_0},x_{i_1}) ,\dots, d(x_{i_{m-1}},x_{i_m}) \|_p, 
\end{equation}
where $\|-\|_p$ is the $\ell_p$-norm, $d$ is the distance and the maximum is taken over all strictly increasing subsequences $i_0< \dots < i_m$ of $0,\dots,n$. 
For $p=1,$ using the triangle inequality, we obtain 
\begin{equation}
\WW_1(x_0,\dots,x_n)=d(x_0,x_1)+\dots + d(x_{n-1},x_n). 
\end{equation}
For $p=\infty,$ the $\ell_p$-norm is equal to the diameter of the set of points in the tuple 
\begin{equation}
\WW_\infty(x_0,\dots,x_n) = \diam(\{x_0,\dots,x_n\}). 
\end{equation}
For any $r>0,$ the $\ell_p$-Vietoris-Rips simplicial set $\VR^p_{<r} X$ is defined so that its $n$-simplices are tuples $(x_0,\dots,x_n)$ such that $\WW_p(x_0,\dots,x_n)<r,$ and the face and degeneracy maps are defined by deletions and doublings. Similarly we  define $\VR^p_{\leq r} X$ by replacing $<$ with $\leq.$

The definition of the  $\ell_p$-Vietoris-Rips complexes $\vr^p_{< r} X$ and $\vr^p_{\leq r} X$ is based on the concept of the $\ell_p$-weight of a finite non-empty subset of a metric space. It is defined by the formula  
\begin{equation}
\ww_p(\{x_0,\dots,x_n\})= \min_{\pi\in \Sigma_{[n]}} \WW_p(x_{\pi(0)},\dots,x_{\pi(n)}),
\end{equation}
where $\Sigma_{[n]}$ denotes the group of permutations of the set $\{0,\dots,n\}.$ 
Thus the $\ell_p$-Vietoris-Rips complex $\vr^p_{<r} X$ consists of finite non-empty subsets of $X$ whose $\ell_p$-weight is less than $r.$ The non-strict version of the simplicial complex $\vr^p_{\leq r} X$ is defined similarly. 

The simplicial complex  $\vr^\infty_{<r}X$ is the classical Vietoris-Rips complex and the geometric realization of the simplicial set $\VR^\infty_{<r} X$ is homotopy equivalent to the geometric realization of $\vr^\infty_{<r} X.$ Moreover, it is easy to see that the $\ell_p$-Vietoris-Rips complex forms a sub-complex in the classical Vietoris-Rips complex
$\vr^p_{<r} X \subseteq \vr_{< r}X$ and $\vr^p_{\leq r} X \subseteq \vr_{\leq r}X.$
The blurred magnitude homology are the homology of the $\ell_1$-Vietoris-Rips simplicial set. 

\subsection{Distance matrix norms} In order to unify the cases of simplicial sets and simplicial complexes, we define a general notion of distance matrix norm. 
For $n\geq 0,$ we denote by $\Dist_{[n]}$ the subset of the set of real square matrices of order $n+1$ consisting of  all  distance matrices of an $(n+1)$-tuple in a metric space. A distance matrix norm $\nu$ consists of maps  $\nu:\Dist_{[n]}\to \RR_{\geq 0},$ for each $n\geq 0,$ satisfying a  list of axioms (Definition \ref{def:SM}). 
The $\nu$-weight of a tuple of points is defined as 
\begin{equation}
\WW_\nu(x_0,\dots,x_n)=\nu(D(x_0,\dots,x_n)),
\end{equation}
where $D(x_0,\dots,x_n)=(d(x_i,x_j))_{i,j}$ is the distance matrix. 
Then the $\nu$-Vietoris-Rips simplicial set $\VR^\nu_{<r}X$ is the simplicial set consisting of tuples whose $\nu$-weight is less than $r.$ 

We construct two families of distance matrix norms $\nu_p$ and $\nu_p^{\sf sym}$ and show that the $\ell_p$-Vietoris-Rips simplicial set is the $\nu_p$-Vietoris-Rips simplicial set; and the geometric realization of the $\ell_p$-Vietoris-Rips complex is homotopy equivalent to the geometric realization of the $\nu_p^{\sf sym}$-Vietoris-Rips simplicial set. Thus, both of the theories can be reduced to the theory of $\nu$-Vietoris-Rips simplicial sets. We also give some other interesting examples of distance matrix norms (Subsection \ref{subsection:cyclic_versions}).

\subsection{Stability theorem} 
For any distance matrix norm $\nu,$ we prove a version of the stability theorem for the ``$\nu$-persistent homology''. 
Namely, we show that the interleaving distance between persistent modules $H_n(\VR_{<*}^{\nu} X)$ and $H_n(\VR_{<*}^{\nu} Y)$ is not greater than $2C_{n+2}(\nu)\cdot d_{GH}(X,Y),$ where $d_{GH}$ is the Gromov-Hausdorff distance and $C_n(\nu)$ denotes some constant  (Theorem \ref{theorem:stability})
\begin{equation}
d_{\sf int}( H_n(\VR_{<*}^{\nu} X), H_n(\VR_{<*}^{\nu} Y)) \leq 2 C_{n+2}(\nu) 
 \cdot   d_{GH}( X,Y ).
\end{equation}
As corollaries, we obtain the inequality for $\ell_p$-Vietoris-Rips simplicial sets
\begin{equation}
 d_{\sf int}( H_n(\VR_{<*}^{p} X), H_n(\VR_{<*}^{p} Y)) \leq 2 
(n+2)^{\frac{1}{p}}\cdot  d_{GH}( X,Y ),
\end{equation}
for $\ell_p$-Vietoris-Rips complexes 
\begin{equation}
 d_{\sf int}( H_n(\vr_{<*}^{p} X), H_n(\vr_{<*}^{p} Y)) \leq 2 
(n+2)^{\frac{1}{p}}\cdot  d_{GH}( X,Y ),   
\end{equation}
and for the blurred magnitude homology 
\begin{equation}
 d_{\sf int}( \MH_{n,< *}(X), \MH_{n,< *} ( Y) ) \leq 2 
(n+2)\cdot  d_{GH}( X,Y ).   
\end{equation}

\subsection{Riemannian manifolds}

Hausmann showed that for a compact Riemannian manifold $M$ and a sufficiently small scale parameter $r>0$ the geometric realization of the Vietoris-Rips complex $\vr_{<r} M$ is homotopy equivalent to $M$ \cite{hausmann1994vietoris}. We generalize this theorem by showing that  the geometric realization of $\VR^\nu_{<r} M$ is homotopy equivalent to $M$ for a sufficiently small $r>0$ and any $\nu.$ In particular, we obtain that the blurred magnitude homology of $M$ is isomorphic to the ordinary homology of $M$ for a sufficiently small $r.$ 
Moreover, we prove a more general version of this theorem for geodesic spaces (Theorem \ref{th:manifold}).

For any $1\leq p \leq \infty,$ we define $r_p(M)$ as the infimum of $r>0$ such that the geometric realization of  $\VR^p_{<r}M$ is not homotopy equivalent to $M.$ We compute this number for the circle of perimeter one $S^1$: 
\begin{equation}
r_p(S^1) = \frac{2^{\frac{1}{p}}}{2+2^{\frac{1}{p}}}.  
\end{equation} In particular, we obtain $r_\infty(S^1)=\frac{1}{3},$ which coincides with the result of 
Adamaszek and Adams   \cite{adamaszek2017vietoris}. 

\subsection{Metric completions and shortly contractible spaces} 
We prove a general theorem (Theorem \ref{th:embedding}) which states that under certain conditions an embedding of metric spaces $A\hookrightarrow X$ induces a weak equivalence of simplicial sets
\begin{equation}
\VR^\nu_{<r} A \overset{\sim}\longrightarrow \VR^\nu_{<r} X    
\end{equation}
for any distance matrix norm $\nu$ and any $r>0.$ As a corollary we show that the homotopy type of the strict $\nu$-Vietoris-Rips simplicial set does not change after taking the metric completion:
\begin{equation}
|\VR^\nu_{<r} X| \simeq  |\VR^\nu_{<r} \hat X|.    
\end{equation}
In particular, we obtain 
$\MH_{n,<r}(X)\cong \MH_{n,<r}(\hat X).$

We say that a metric space $X$ is \emph{shortly contractible} if there exists a contracting homotopy to the point $h:X\times [0,c]\to X,$ which is a short map (Definition \ref{def:shortly_contractible}). For example, bounded convex subsets, bounded star domains and the semicircle are shortly contractible.  As a corollary of the general theorem, we obtain that for a shortly contractible metric space  $X$ the simplicial set $\VR^\nu_{<r}X$ is weakly contractible. In particular, we have  
 $\MH_{n,<r}(X)=0$ for $n\geq 1.$ 

Note that these results hold only for the strict versions of the simplicial sets $\VR^\nu_{<r} X$ and do not hold for the non-strict versions $\VR^\nu_{\leq r}X.$ There is an example of a metric space $X$ such that 
$\MH_{1,\leq r}(X)\not\cong \MH_{1,\leq r}(\hat X)$ (Remark \ref{remark:non-iso-for-closed}), and there is also an example of a shortly contractible metric space $X$ such that $\MH_{1,\leq r}(X)\neq 0$ 
(Remark \ref{remark:shortly_contr}).

\subsection{Pro-simplicial sets and limit homology} 

Despite the fact that for different distance matrix norms $\nu$ the simplicial sets $\VR^\nu_{<r} X$ are usually different, all of them are similar for small $r.$ This observation can be formalized in terms of pro-simplicial sets, which can be understood as ``formal cofiltered limits'' of simplicial sets. 

We denote by $\VR^{\nu,n}_{<r} X$ the $n$-skeleton of the $\nu$-Vietoris-Rips simplicial set. 
Then the functor 
$(0,\infty)\to {\sf sSet}  $ 
sending $r$ to $\VR^{\nu,n}_{<r}X$ defines a pro-simplicial set that we denote by $\VR^{\nu,n}_{\sf pro} X.$ We prove that this pro-simplicial set does not depend on $\nu$
\begin{equation}
\VR^{\nu,n}_{\sf pro} X \cong \VR^{\infty,n}_{\sf pro} X.     
\end{equation}
In particular, we obtain that the limits of homology groups coincide
\begin{equation}
 \lim_{r\to 0} H_n(\VR^\nu_{<r} X)\cong \lim_{r\to 0} H_n(\vr_{<r} X).    
\end{equation}
 
\subsection{Filtered colimits of metric spaces} 

A disadvantage of the ordinary magnitude homology is that it does not commute with filtered colimits of metric spaces and short maps. However, we show that the strict versions of the $\nu$-Vietoris-Rips simplicial sets $\VR^\nu_{<r}X$ commute with filtered colimits of metric spaces up to homotopy. More precisely, for any filtered category $I$ and any $I$-diagram $X_i$ in the category of metric spaces and short maps the map
\begin{equation}
\colim (\VR^\nu_{<r} X_i) \overset{\sim}\longrightarrow  \VR^\nu_{<r} (\colim X_i) 
\end{equation}
is a weak equivalence (Theorem  \ref{th:filtered_colimits}). In particular, the strictly blurred magnitude homology commutes with filtered colimits of metric spaces and short maps
\begin{equation}
\colim \MH_{n,<r}(X_i)\cong \MH_{n,<r}(\colim X_i).
\end{equation}

\subsection{Acknowledgements}

We are grateful to Emily Roff for useful discussions about filtered colimits of metric spaces (see Section 
\ref{sec:filtered_colimits}).

\section{Vietoris-Rips spaces} 

This section is devoted to unifying the notion of ``Vietoris-Rips spaces'' through the concept of distance matrix norms. We define a distance matrix norm, and for any distance matrix norm $\nu,$ left infinite interval $L\subseteq \RR$ and metric space $X,$ we associate a simplicial set $\VR^\nu_L X.$ 
We will show that the study of both $\ell_p$-Vietoris-Rips simplicial sets and $\ell_p$-Vietoris-Rips complexes can be reduced to the study of the simplicial set $\VR^\nu_L X$ for a general $\nu,$ and the blurred magnitude homology theory can be thought of as the study of the $\ell_1$-Vietoris-Rips simplicial set. 
Throughout the article, we denote by $X,Y$ some metric spaces, by $p$ an element of $[1,\infty],$ and by $L\subseteq \RR$ a left infinite interval, which is one of the sets $(-\infty,r),(-\infty,r]$ or $\RR.$  

\subsection{Distance matrix norms} For $n\geq 0,$ we denote by $\Mat_{[n]}(\RR)$ the set of real square matrices $A=(A_{i,j})_{i,j\in [n]}$ of order $n+1$, whose entries are indexed by the set $[n]=\{0,\dots,n\}.$ 
Consider a simplicial set $\Mat_{[\bullet]}(\RR),$ whose $n$-th component is 
$\Mat_{[n]}(\RR),$
the face map $\partial_i$ is defined by deleting the $i$-th row and the $i$-th column of a matrix, and the degeneracy map $s_i$ is defined by doubling the $i$-th row and the $i$-th column of a matrix. In other words, we can say that the simplicial set  $\Mat_{[\bullet]} (\RR): \Delta^{op} \to {\sf Set}$ is defined so that, its $n$-simplices are maps $A:[n]\times [n]\to \RR$ and for a monotone map $f:[m]\to [n],$ the map $f^*:\Mat_{[n]}(\RR)\to \Mat_{[m]}(\RR)$ defined by $f^*A=A (f\times f).$ 
We consider a simplicial subset of  ``distance matrices''
\begin{equation}
 \Dist_{[\bullet]} \subseteq \Mat_{[\bullet]}(\RR)
\end{equation}
consisting of matrices $D$ such that $D_{i,j}\geq 0,$ $D_{i,i}=0,$ $D_{i,j}=D_{j,i}$ and $D_{i,j}+D_{j,k}\geq D_{i,k}$ for $i,j,k\in [n].$ Here we allow $D_{i,j}=0$ for $i\neq j.$ In other words $\Dist_{[n]}$ consists of all  pseudo-metrics on the set $\{0,\dots,n\}.$ It is easy to see that $\Dist_{[\bullet]}$ is a simplicial subset of $\Mat_{[\bullet]}(\RR)$. Moreover, for each $n$ the set $\Dist_{[n]}$ is a convex cone in $\Mat_{[n]}(\RR)$ i.e. for any $\alpha,\alpha'\in \RR_{\geq 0}$ and $D,D'\in \Dist_{[n]}$ we have $\alpha D+\alpha'D'\in \Dist_{[n]}.$

\begin{definition}\label{def:SM}
By a \emph{distance matrix norm}, we mean a collection of maps 
\begin{equation}
\nu:\Dist_{[n]} \to \RR_{\geq 0}, \hspace{1cm} n\geq 0 
\end{equation}
 such that for any $D,D'\in \Dist_{[n]}$ and any $\alpha\in \RR_{\geq 0},$ we have
\begin{enumerate} 
    \item  $\nu(\alpha\cdot  D)=\alpha\cdot \nu(D);$
    \item $\nu(D+D')\leq \nu(D)+\nu(D');$
    \item If $D_{i,j}\leq D'_{i,j}$ for all $i,j\in [n],$ then $\nu(D)\leq \nu(D');$
    \item $\nu( \begin{smallmatrix}
       0 & 1 \\
       1 & 0
    \end{smallmatrix} )=1;$
    \item  
    $\nu(\partial_i D) \leq  \nu(D)$  and $ \nu(s_i D) = \nu(D).$
\end{enumerate}
\end{definition} 
The property (3) is called monotone; the property (4) is called normalization; the property (5) is called simpliciality. 

We denote by $E_n\in \Dist_{[n]}$  the matrix such that $(E_n)_{i,j}=1,$ for any $i\neq j,$ and $(E_n)_{i,i}=0,$ for any $i.$ For a distance matrix norm $\nu,$ we consider an increasing sequence of constants $C_n(\nu)$
\begin{equation}
1=C_1(\nu)\leq C_2(\nu) \leq C_3(\nu)\leq  \dots 
\end{equation}
defined by the formula
\begin{equation}
C_n(\nu) = \nu(E_n).
\end{equation}
The sequence is increasing because $\partial_0 E_{n+1}=E_n$ and $\nu(\partial_0D)\leq \nu(D).$

Let $\Sigma_{[n]}$ be the group of permutations of $[n].$ For a matrix $D\in \Dist_{[n]}$ and $\pi\in \Sigma_{[n]},$ we denote by $D^\pi$ a matrix such that $D^\pi_{i,j}=D_{\pi(i),\pi(j)}.$ A distance matrix norm $\nu$ is called \emph{symmetric} if 
\begin{equation}
\nu(D)=\nu(D^\pi) 
\end{equation}
for all $D\in \Dist_{[n]}$ and  $\pi\in \Sigma_{[n]}$.
For any distance matrix norm $\nu,$ we define a collection of maps $\nu^{\sf sym}:\Dist_{[n]}\to \RR$ as 
\begin{equation}
\nu^{\sf sym}(D) = \min_{\pi\in \Sigma_{[n]}} \nu(D^\pi).
\end{equation}

\begin{lemma} For any distance matrix norm $\nu,$ the collection of maps 
$\nu^{\sf sym}$ is a symmetric distance matrix norm. 
\end{lemma}
\begin{proof}
The properties (1)-(4) are obvious. The property  $\nu^{\sf sym}(\partial_i D) \leq \nu^{\sf sym}(D)$ is also obvious. Let us prove that for any matrix $D\in \Dist_{[n]},$ we have $\nu^{\sf sym}(s_i D)= \nu^{\sf sym}(D).$ Since $\partial_i s_iD=D,$ we have $\nu^{\sf sym}(s_i D) \geq  \nu^{\sf sym}(D).$ 

Let us prove that $\nu^{\sf sym}(s_i D) \leq  \nu^{\sf sym}(D).$ Here we will treat the matrix $D$ as a map $D:[n]\times [n]\to \RR$ and any map $f:[m]\to [n]$ we set $D^f=D(f\times f).$  Then $\partial_k D = D^{\partial^k}$ and $s_kD=D^{s^k},$ where $\partial^k : [n-1]\to [n]$ and $s^k:[n+1]\to [n]$ are the coface and  codegeneracy maps. Take $\pi\in \Sigma_{[n]}$ such that  $\nu^{\sf sym}(D)=\nu(D^\pi).$ Then there exists $\pi'\in  \Sigma_{[n+1]}$ such that 
$s^{i}   \pi' = \pi s^{\pi^{-1}(i)}.$
Indeed, it can be defined by the formula
\begin{equation}
 \pi'(k) = 
 \begin{cases}
 \partial^{i+1}   \pi     s^{\pi^{-1}(i)}(k), & k\neq \pi^{-1}(i)+1,\\
 i+1, & k=\pi^{-1}(i)+1.
 \end{cases}
\end{equation}
Then we obtain 
$\nu^{\sf sym}( D^{s^i} )
\leq \nu(D^{s^i  \pi'}) = \nu(D^{\pi  s^{\pi^{-1}(i)}} )\   =\nu(D^{\pi}),$
and hence $\nu^{\sf sym}( s_iD )\leq \nu^{\sf sym}(D).$
\end{proof}

Two main examples of distance matrix norms for us are the following. For $p\in [1,\infty],$ we set 
\begin{equation}
\nu_p(D) = \max_{0\leq i_0<\dots<i_m\leq n} \| D_{i_0,i_1},D_{i_1,i_2},\dots, D_{i_{m-1},i_m} \|_p,
\end{equation}
where $\|-\|_p$ denotes the $\ell_p$-norm. Since the diagonal of a distance matrix $D$ is trivial, we have $(s_iD)_{i,i+1}=0.$ Using this we can prove that $\nu_p(s_iD)=\nu_p(D).$ All other conditions of a distance matrix norm are obvious. The second example that we need is the symmetric version of this distance matrix norm  
\begin{equation}
    \nu_p^{\sf sym}(D) = \min_{\pi\in \Sigma_{[n]}} \nu_p(D^\pi).
\end{equation}
The sequences of constants $C_n$ for these examples are  
\begin{equation}
C_n(\nu_p) =C_n(\nu_p^{\sf sym}) =  n^{\frac{1}{p}},  
\end{equation}
where $n^{\frac{1}{\infty}}=1.$

\subsection{Vietoris-Rips simplicial sets}

Let $X$ be a metric space and $\nu$ be a distance matrix norm. For  a tuple of points $(x_0,\dots,x_n)\in X^{n+1},$ we consider the distance matrix  $D(x_0,\dots,x_n)=(d(x_i,x_j))_{i,j\in [n]}.$ Then the \emph{$\nu$-weight} of the tuple is defined as  
\begin{equation}
\WW_\nu(x_0,\dots,x_n) = \nu(D(x_0,\dots,x_n)).
\end{equation}
\begin{lemma}\label{lemma:W_nu} Let $\nu$ be a distance matrix norm,  $\sigma=(x_0,\dots,x_n)\in X^{n+1}$ be a tuple and $0\leq i\leq n$. Set $\tilde \sigma=\{x_0,\dots,x_n\}.$ Then the following holds.  
\begin{enumerate}
    \item 
$\WW_\nu(\partial_i \sigma) \leq  \WW_\nu(\sigma)$ and 
$\WW_\nu(s_i\sigma) = \WW_\nu(\sigma).$
\item $\WW_\nu(x_0,x_1) = d(x_0,x_1).$
\item $\WW_\nu(\sigma)\leq \WW_\nu(x_0,\dots,x_i)+\WW_\nu(x_i,\dots,x_n).$
\item $\WW_\nu(\sigma)\leq \sum_j d(x_j,x_{j+1}).$
\item $\diam(\tilde \sigma) \leq  \WW_\nu(\sigma) \leq C_n(\nu)\cdot \diam(\tilde \sigma).$
\end{enumerate}
\end{lemma}
\begin{proof}
(1). Follows from the simpliciality of $\nu.$

(2). Follows from $\nu(D(x_0,x_1) ) = \nu(d(x_0,x_1) \cdot E_1)=d(x_0,x_1)\cdot \nu(E_1)=d(x_0,x_1).$

(3). Consider a matrix $A$ defined by 
\begin{equation}
A_{k,l} = 
\begin{cases}
d(x_k,x_l), & k,l \leq i, \text{ or } i\leq k,l, \\ 
d(x_k,x_i)+d(x_i,x_l), & k<i<l, \text{ or } l<i<k.
\end{cases}    
\end{equation}
Then $D(\sigma)_{k,l}\leq A_{k,l}$ and 
\begin{equation}
A= s_n\dots s_{i+1}s_{i} D(x_0,\dots,x_i) + s_0\dots s_0D(x_i,\dots,x_n).    
\end{equation}
In particular, we obtain  $A\in \Dist_{[n]}.$
Then $\nu(D(\sigma))\leq \nu(A)\leq \nu(D(x_0,\dots,x_i))+\nu(D(x_i,\dots,x_n)).$

(4). Follows from (2) and (3).

(5). (1) and (2) imply that $d(x_k,x_l)=\WW_\nu(x_k,x_l)\leq \WW_\nu(\sigma).$ Then $\diam(\tilde \sigma) \leq  \WW_\nu(\sigma).$  On the other hand $D(\sigma)_{i,j}\leq \diam(\tilde \sigma)\cdot (E_n)_{i,j}.$ Then $\WW_\nu(\sigma) \leq C_n(\nu)\cdot \diam(\tilde \sigma).$
\end{proof}

For a left infinite interval $L\subseteq \RR$ (which is either $L=(-\infty,r],$ or $L=(-\infty,r),$ or $L=\RR$),  we define the  $\nu$-Vietoris-Rips simplicial set $\VR^\nu_L X$ so that its $n$-simplices are defined by the formula   
\begin{equation}
(\VR^\nu_L X)_n = \{(x_0,\dots,x_n)\mid \WW_\nu(x_0,\dots,x_n) \in L\},
\end{equation}
and face and degeneracy maps are defined by the formulas 
\begin{equation}
\begin{split}
\partial_i(x_0,\dots,x_n) &= (x_0,\dots,x_{i-1},x_{i+1},\dots,x_n), \\ 
s_i(x_0,\dots,x_n) &= (x_0,\dots,x_{i},x_{i},\dots,x_n). 
\end{split}
\end{equation}
Lemma \ref{lemma:W_nu}(1) implies that this simplicial set is well defined. We will also use the following notations
\begin{equation}
\VR^\nu_{<r} X = \VR^\nu_{(-\infty,r)} X, \hspace{1cm} \VR^\nu_{\leq r} X = \VR^\nu_{(-\infty,r]} X, 
\end{equation}
and the following notation for the $n$-skeleton 
\begin{equation}
    \VR^{\nu,n}_L X = \sk_n( \VR^\nu_L X).
\end{equation}

The $\ell_p$-weight of a tuple $\sigma=(x_0,\dots,x_n)$ of points of $X$ is defined as  $\WW_p(\sigma)=\WW_{\nu_p}(\sigma).$ Therefore 
\begin{equation}
    \WW_p(x_0,\dots,x_n) = \underset{0\leq i_0< \dots < i_m\leq n }\max \| d(x_{i_0},x_{i_1}) ,\dots, d(x_{i_{m-1}},x_{i_m}) \|_p. 
\end{equation} 
Then the $\ell_p$-Vietoris-Rips simplicial set is defined by $\VR^{p}_L X = \VR^{\nu_p}_L X.$ 
Note that for $p=1,$ we have
$\WW_1(x_0,\dots,x_n) = d(x_0,x_1)+\dots + d(x_{n-1},x_n).$
Therefore 
\begin{equation}\label{eq:VR^1}
(\VR^1_L X)_n = \{(x_0,\dots,x_n) \mid d(x_0,x_1)+\dots+d(x_{n-1},x_n)\in L \}.   
\end{equation}
For $p=\infty,$ we have
$W_\infty(x_0,\dots,x_n) = \diam(\{x_0,\dots,x_n\}).$ 
Therefore we obtain 
\begin{equation}
(\VR^\infty_L X)_n = \{(x_0,\dots,x_n) \mid \diam (\{x_0,\dots,x_n\}) \in L \}.
\end{equation}
Lemma \ref{lemma:W_nu}(4,5) implies that for any distance matrix norm $\nu,$ there are inclusions
\begin{equation}\label{eq:inclusion-to-infty}
\VR^1_{L}X \subseteq \VR^\nu_L X \subseteq \VR^\infty_L X, \hspace{1cm} 
\VR^{\infty,n}_{C_n(\nu)^{-1} \cdot L}  X \subseteq  \VR^{\nu,n}_L X.
\end{equation}

\subsection{Relation to blurred magnitude homology} 

The $\ell_1$-Vietoris-Rips simplicial set $\VR^1_{\leq r} X$ is the simplicial set used in the blurred magnitude homology theory. In 
\cite[Def. 15 (4)]{otter2018magnitude} (see also \cite{govc2021persistent}, \cite{cho2019quantales}) for a metric space $X,$ they define its $[0,\infty]^{op}$-nerve as 
\begin{equation}
N(X)(r)_n = \{ (x_0,\dots,x_n)\mid d(x_0,x_1) + \dots + d(x_{n-1},x_n) \leq r \}
\end{equation}
and define the blurred magnitude homology of $X$ as the homology of this persistent simplicial set 
$\MH_{n,\leq r}(X) = H_n(N(X)(r)).$ 
Using \eqref{eq:VR^1}, we obtain $\VR^1_{\leq r} X = N(X)(r)$ and 
\begin{equation}
\MH_{n,\leq r}(X) = H_n(\VR^1_{\leq r} X).
\end{equation} 

The magnitude homology of $X$ can be defined as the reduced homology of the quotient simplicial set
\begin{equation}
\MH_{n,r} (X) = \tilde H_n(\VR^1_{\leq r} X/\VR^1_{< r} X).
\end{equation}
We will also consider the strict version of the blurred magnitude homology
\begin{equation}
\MH_{n,<r}(X) = H_n(\VR^1_{<r} X).
\end{equation}
This strict version will be called strictly blurred magnitude homology, and the non-strict version will be called non-strictly blurred magnitude homology.  It is easy to see that there is a long exact sequence connecting all these three types of magnitude homology
\begin{equation}\label{eq:long_exact_magnitude}
\dots \to \MH_{n+1,r}(X)  \to     \MH_{n,<r}(X) \to \MH_{n,\leq r}(X) \to \MH_{n,r}(X) \to \dots.  
\end{equation}

\subsection{Vietoris-Rips complexes}

Assume that $\nu$ is a symmetric distance matrix norm, and $X$ is a metric space. Since $\nu$ is symmetric, we have 
\begin{equation}
\WW_\nu(x_0,\dots,x_n)=\WW_\nu(x_{\pi(0)},\dots,x_{\pi(n)})   \end{equation}
for any $\pi\in \Sigma_{[n]}.$ Let $\{x_0,\dots,x_n\}\subseteq X$ be a finite set such that $x_0,\dots,x_n$ are distinct points. We define its $\nu$-weight as
\begin{equation}
\ww_\nu(\{ x_0,\dots,x_n \})= \WW_\nu(x_0,\dots,x_n).
\end{equation} 
The definition does not depend on the choice of the order of the points because $\nu$ is symmetric.  

\begin{lemma}\label{lemma:W}
For any sequence $x_0,\dots,x_n\in X$ (with possible repetitions), we have 
\begin{equation}
\ww_\nu(\{ x_0,\dots,x_n \}) = \WW_\nu(x_0,\dots,x_n). 
\end{equation}
\end{lemma}
\begin{proof} Set $\sigma=\{x_0,\dots,x_n\}.$
If all points $x_0,\dots,x_n$ are distinct, then it is obvious. If they are not distinct, there exists a permutation $\pi\in \Sigma_{[n]}$ and a subsequence $0\leq i_0<\dots<i_m\leq n$ such that the points  $x_{\pi(i_0)},\dots,x_{\pi(i_m)}$ are distinct, $\sigma = \{x_{\pi(i_0)},\dots,x_{\pi(i_m)} \}$ and for $i_{k} \leq j< i_{k+1}$ we have $x_{\pi(i_k)}=x_{\pi(j)}.$ Then $(x_{\pi(0)},\dots,x_{\pi(n)})$ 
can be obtained from $(x_{\pi(i_0)},\dots,x_{\pi(i_m)})$ by applying degeneracy maps. 
It follows that  $\WW_\nu(x_{\pi(0)},\dots,x_{\pi(n)})= \WW_\nu(x_{\pi(i_0)},\dots,x_{\pi(i_m)})=\ww_\nu(\sigma).$
\end{proof}

The $\nu$-Vietoris-Rips complex $\vr^\nu_L X$ is defined as a simplicial complex consisting of finite non-empty subsets $\sigma\subseteq X$ such that $\ww_\nu(\sigma)\in L$
\begin{equation}
\vr^\nu_L X =\{  \sigma \underset{\scalebox{0.6}{\sf f.n.}}\subseteq X \mid \ww_\nu(\sigma) \in L  \}. 
\end{equation}
We will also consider its $n$-skeleton 
\begin{equation}
\vr^{\nu,n}_L X \cong \sk_n( \vr^\nu_L X ).
\end{equation}

\begin{proposition}\label{prop:vrVR-equiv}
For a symmetric distance matrix norm $\nu,$ a metric space $X$ and any $n\geq 0,$ there are homotopy equivalences  
\begin{equation}
    |\vr^\nu_L X| \simeq |\VR^\nu_L X|, \hspace{1cm} |\vr^{\nu,n}_L X| \simeq | \VR^{\nu,n}_L X|,
\end{equation}
which are natural with respect to inclusions  defined by inclusions $L\subseteq L' $ and $X\subseteq X'.$ 
\end{proposition}
\begin{proof} In this proof we use the results and notations of Appendix ``Simplicial sets vs. simplicial complexes'' (Section \ref{sec:ss_sc}). In particular, for a simplicial complex $K$ we denote by $\SS(K)$ the simplicial set, whose $n$-simplices are tuples $(x_0,\dots,x_n)$ such that $\{x_0,\dots,x_n\}\in K$. Using Lemma \ref{lemma:W}, we obtain $\VR^\nu_L X=\SS(\vr^\nu_L X)$ and $\VR^{\nu,n}_L X=\SS(  \vr^{\nu,n}_L X).$ Then the statement follows from Theorem \ref{th:geometric:realisation:SSK}.  
\end{proof}

The $\ell_p$-weight of a non-empty finite subset  $\tau = \{x_0,\dots,x_n\}\subseteq X,$ where $x_0,\dots,x_n$ are distinct points, is defined by 
\begin{equation}
\ww_p(\tau) = \min_{\pi\in \Sigma_{[n]}} \WW_p( x_{\pi(0)},\dots,x_{\pi(n)} ).
\end{equation}
For $p=1$ and $p=\infty,$ we have 
\begin{equation}
\ww_1(\tau) = \min_{\pi\in \Sigma_{[n]}} d(x_{\pi(0)},x_{\pi(1)}) + \dots + d(x_{\pi(n-1)},x_{\pi(n)}) \end{equation}
and $\ww_\infty(\tau) = \diam(\tau).$

The $\ell_p$-Vietoris-Rips complex is defined by $\vr^p_L X = \vr^{\nu}_L X$ for $\nu=\nu^{\sf sym}_p.$ Therefore 
\begin{equation}
\vr^p_L X = \{ \tau \underset{\scalebox{0.6}{\sf f.n.}}\subseteq X\mid \ww_p(\tau)\in L \}.
\end{equation}
Since $\ww_\infty(\tau)=\diam(\tau),$ we obtain that $\vr^\infty_L X$ is the classical Vietoris-Rips complex
\begin{equation}
\vr^\infty_L X = \{  \tau \underset{\scalebox{0.6}{\sf f.n.}}\subseteq X\mid \diam(\tau)\in L \}.
\end{equation}
Let us denote by $\VR^{{\sf sym}\text{-}p}_L X$ the Vietoris-Rips simplicial set associated with the symmetric distance matrix norm $\nu_p^{\sf sym}$ 
\begin{equation}
(\VR^{{\sf sym}\text{-}p}_L X)_n = \{ (x_0,\dots,x_n) \mid \ww_p(\{x_{0},\dots,x_{n}\})\in L \}.
\end{equation}
Then Proposition \ref{prop:vrVR-equiv} implies that there are  homotopy equivalences between the geometric realisations and their $n$-skeletons 
\begin{equation}
|\vr^p_L X| \simeq |\VR^{{\sf sym}\text{-}p}_L X|, \hspace{1cm} |\vr^{p,n}_L X| \simeq |\VR^{{\sf sym}\text{-}p,n}_L X|,    
\end{equation} 
which are natural with respect to inclusions defined by inclusions $L\subseteq L'$ and $X\subseteq X'.$
Since $\nu_\infty=\nu_\infty^{\sf sym},$ we also have 
\begin{equation}
|\vr^\infty_L X| \simeq |\VR^\infty_L X|. 
\end{equation}

Note that for any symmetric distance matrix norm $\nu,$ Lemma \ref{lemma:W_nu} implies that 
\begin{equation}
\vr^1_L X \subseteq \vr^\nu_L X \subseteq \vr^\infty_L X.
\end{equation}

\subsection{Cyclic versions of the Vietoris-Rips spaces}\label{subsection:cyclic_versions}

Let us give some other examples of distance matrix norms and the associated simplicial sets and simplicial complexes. For a matrix $D\in \Dist_{[n]},$ we set 
\begin{equation}
\nu^{\sf c}_p(D) =2^{- \frac{1}{p} } \cdot  \max_{0\leq i_0<\dots <i_m\leq n} \|D_{i_0,i_1},\dots,D_{i_{m-1},i_{m}},D_{i_m,i_0}\|_p.
\end{equation}
Then $\nu^{\sf c}_p$ is another example of a distance matrix norm. The associated $\nu^{\sf c}_p$-weight is denoted by $\WW_p^{\sf c}.$ Using the triangle inequality, for $p=1$ and $\sigma=(x_0,\dots,x_n)$  we obtain
\begin{equation}
\WW_1^{\sf c}(\sigma) = \frac{1}{2} (d(x_0,x_1)+\dots +d(x_{n-1},x_n)+d(x_0,x_n)).
\end{equation}
The corresponding $\nu_p^{\sf c}$-Vietoris-Rips simplicial sets can be called cyclic $\ell_p$-Vietoris-Rips simplicial sets
$\VR^{{\sf c}\text{-}p}_{<r} X.$ Of course, we can take the symmetric version of this distance matrix norm
\begin{equation}
\nu_p^{\sf c.sym}(A) = \min_{\pi\in \Sigma} \nu_p^{\sf c}(A^\pi),
\end{equation}
the associated weight of a finite subset $\tau=\{x_0,\dots,x_n\}\subseteq X $
\begin{equation}
\ww^{\sf c}_p(\tau) = \min_{\pi\in \Sigma_{[n]}} \WW^{\sf c}_p(x_{\pi(0)},\dots,x_{\pi(n)})
\end{equation}
and consider the associated \emph{cyclic $\ell_p$-Vietoris-Rips complex} $\vr^{{\sf c}\text{-}p}_{<r} X$ that consists of non-empty finite subsets $\tau\subseteq X$ such that $\ww^{\sf c}_p(\tau)<r.$

\subsection{Cho's definition} 
We fix some $p\in [1,\infty]$  and consider the binary operation for non-negative real numbers 
\begin{equation}
  a\otimes  b = \|a,b\|_p.  
\end{equation}
If $p<\infty,$ we have $a\otimes  b=(a^p+b^p)^{\frac{1}{p}},$ and for $p=\infty,$ we have $a\otimes  b = \max(a,b).$ This binary operation defines a structure of a monoid on $\RR_{\geq 0}$ such that 
\begin{equation}
a_1 \otimes  \dots \otimes  a_n = \|a_1,\dots,a_n\|_p. 
\end{equation}

For a metric space $X,$ Cho defines the $\ell_p$-Vietoris-Rips simplicial set $\widetilde{\VR}^p_{\leq r} X$ of a metric space $X$ as a simplicial set, whose $n$-simplices are tuples $(x_0,\dots,x_n)$ such that there exists a sequence of non-negative real numbers $r_1,\dots,r_n$ satisfying the following conditions: $r_1\otimes \dots \otimes r_n \leq r $ and $d(x_i,x_j)\leq r_{i+1} \otimes \dots \otimes r_{j}$ for $i\leq j$ (see 
\cite[Prop. 13]{cho2019quantales}).

\begin{proposition}\label{prop:Cho} The definition of Cho is equivalent to our definition
\begin{equation}
\widetilde{\VR}^p_{\leq r} X = \VR^p_{\leq r} X.
\end{equation} 
\end{proposition}
\begin{proof}
First we prove that $\widetilde{\VR}^p_{\leq r} X \subseteq  \VR^p_{\leq r} X.$ Take  $(x_0,\dots,x_n)\in (\widetilde{\VR}^p_{\leq r} X)_n.$ Then there exist $r_1,\dots,r_n$ such that $r_1\otimes \dots \otimes r_n\leq r$ and 
$d(x_i,x_j) \leq r_{i+1}\otimes \dots \otimes r_{j}$ for $i\leq j.$ Then for any subsequence $0\leq i_0<\dots <i_m\leq n,$ we have $d(x_{i_0},x_{i_1})\otimes\dots \otimes d(x_{i_{m-1}},x_{i_m}) \leq r_{i_0+1}\otimes \dots \otimes r_{i_m} \leq r.$ Therefore $\WW_p(x_0,\dots,x_n)\leq r.$

Now we prove that $\VR^p_{\leq r} X\subseteq  \widetilde{\VR}^p_{\leq r} X.$ For any $n$-simplex $(x_0,\ldots,x_n)\in (\VR^p_{\leq r}X)_n$, we define $r_1,\ldots,r_n$ inductively. First we set $r_1=d(x_0,x_1)$. Suppose we have defined $r_1,\ldots,r_{k-1}$. We define $r_{k}$ to be the least number such that the following inequalities are satisfied:
\begin{equation}
d(x_i,x_{k}) \leq r_{i+1}\otimes\cdots\otimes r_{k} \hspace{5mm} 0\leq i\leq k-1.
\end{equation}
Since it is the least number, there exists an index $f(k)< k$ such that the equality holds
\begin{equation}
d(x_{f(k)},x_{k}) = r_{f(k)+1}\otimes\cdots\otimes r_{k}.
\end{equation}

Now we verify that such $r_1,\ldots,r_n$ satisfy the desired properties. Since $0\leq f(k)<k,$ there exists $m$ such that $f^m(n)=0.$ We consider the subsequence $0\leq i_0<\dots < i_m\leq n$  defined by $i_j=f^{m-j}(n).$ Then 
\begin{equation}
d(x_{i_j},x_{i_{j+1}}) = r_{i_j+1}\otimes \dots \otimes r_{i_{j+1}}.
\end{equation}
It follows that 
\begin{equation}
r_1\otimes\cdots\otimes r_n=d(x_{i_0},x_{i_1})\otimes\cdots\otimes d(x_{i_{m-1}},x_{i_m})\leq r,
\end{equation}
and for any $0\leq i<j\leq n$, we know from the construction of $r_j$ that 
\begin{equation}
    d(x_i,x_j)\leq r_{i+1}\otimes\cdots\otimes r_j.
\end{equation}
Therefore we complete the proof.
\end{proof}

\section{Lipschitz and short maps}

In this section we introduce the notion of an $L$-locally $(a,b)$-Lipschitz map $f:X\to Y,$ and an $L$-locally $(a,b)$-Lipschitz pair of maps $f,g:X\to Y.$ We  show that an $L$-locally $(a,b)$-Lipschitz map $f$ induces a morphism of $\nu$-Vietoris-Rips simplicial sets with appropriate scale parameters. Moreover, if $f,g$ is an $L$-locally $(a,b)$-Lipschitz pair, then these morphisms are simplicially homotopic.  

Throughout this section  $X,Y$ denote two metric spaces;  $L$ denotes a left infinite interval; $\nu$ denotes a distance matrix norm. We also set $C_n=C_n(\nu),$ $\WW=\WW_\nu$ and assume that $a,b\geq 0$ are two non-negative real numbers. 

A map between metric spaces $f:X\to Y$ is called \emph{$L$-locally $(a,b)$-Lipschitz}, if for any $x,x'\in X$ we have 
\begin{equation}
d(x,x')\in L \hspace{5mm}  \Rightarrow \hspace{5mm}  d(f(x),f(x')) \leq a \cdot  d(x,x')+b.
\end{equation} 
If $a=1,$ then $f$ is called a $L$-locally $b$-short map. Note that if $b\neq 0,$ the maps are not necessarily continuous. 

A pair of maps $f,g:X\to Y$ is called \emph{$L$-locally $(a,b)$-Lipschitz pair}, if both of them are $L$-locally $(a,b)$-Lipschitz and, for any $x,x'\in X$, we have 
\begin{equation}
d(x,x')\in L \hspace{5mm}  \Rightarrow \hspace{5mm}  d(f(x),g(x')) \leq a \cdot  d(x,x')+b.
\end{equation}
If $a=1,$ then $f,g$ is called an $L$-locally $b$-short pair. If $L=\RR$, we use terms ``$(a,b)$-Lipschitz map'', ``$b$-short map'', ``$(a,b)$-Lipschitz pair'' and ``$b$-short pair''.

\begin{lemma}\label{lemma:main_inequality}
Let $(x_0,\dots,x_n)\in X^{n+1}$  and $(y_0,\dots,y_n)\in Y^{n+1}$ be tuples of points, and $a,b\geq 0$ be real numbers such that  
\begin{equation}
d(y_i,y_j) \leq a\cdot  d(x_i,x_j)+b, \hspace{1cm} 0\leq i,j\leq n.
\end{equation}
Then 
\begin{equation}
\WW(y_0,\dots,y_n)\leq a\cdot \WW(x_0,\dots,x_n)+b\cdot  C_n.
\end{equation}
\end{lemma}
\begin{proof}
The conditions imply that $D(y_0,\dots,y_n)_{i,j}\leq a\cdot D(x_0,\dots,x_n)_{i,j}+b\cdot (E_n)_{i,j}.$ Therefore, using the properties of a distance matrix norm, we obtain 
\begin{equation}
\nu(D(y_0,\dots,y_n)) \leq a\cdot \nu(D(x_0,\dots,x_n))+b\cdot \nu(E_n).
\end{equation}
The statement follows. 
\end{proof}

For a simplicial complex $K,$ we denote by $\SS(K)$ the simplicial set consisting of tuples $(x_0,\dots,x_n)$ such that $\{x_0,\dots,x_n\} \in K.$ 
Basic properties of the functor $\SS:\SCpx\to \sSet$ can be found in Section \ref{sec:ss_sc}. For a set $S,$ we denote by $\Delta(S)$ the simplicial complex of all finite non-empty subsets of $S.$ Therefore  $\SS(\Delta(S))$ is a simplicial set such that $\SS(\Delta(S))_n=S^{n+1}.$ Note that 
\begin{equation}
\VR^\nu_{L} X \subseteq \SS(\Delta(X)).
\end{equation}

\begin{proposition}\label{prop:c-short:induces} For any $n\geq 0,$  an $L$-locally $(a,b)$-Lipschitz map $f:X\to Y$ induces a morphism of simplicial sets
\begin{equation}
f_* : \VR^{\nu,n}_L X \longrightarrow \VR^{\nu,n}_{a\cdot L+b\cdot C_n} Y.
\end{equation}
\end{proposition}
\begin{proof} For $n=0,$ the statement is obvious, so let us assume that $n\geq 1.$
For a tuple $(x_0,\dots,x_n)\in (\VR^\nu_L X)_n,$ we have  $ (x_i,x_j)\in (\VR^\nu_LX)_1$ for $i<j.$ Hence $d(x_i,x_j)=\WW(x_i,x_j)\in L$ and we obtain $d(f(x_i),f(x_j))\leq a\cdot d(x_i,x_j)+b.$ Then by Lemma \ref{lemma:main_inequality}, we obtain $(f(x_0),\dots,f(x_n))\in (\VR^\nu_{a\cdot L+b\cdot C_n} Y)_n.$ Since all simplices of $\VR^{\nu,n}_L X$ can be obtained from $n$-simplices of $\VR^\nu_L X$ by applying face and degeneracy maps, we obtain that the map $f^*:\SS(\Delta(X))\to \SS(\Delta(Y))$  induces a map $\VR^{\nu,n}_L X \longrightarrow \VR^{\nu,n}_{a\cdot L+b\cdot C_n} Y.$ 
\end{proof}

Recall that a simplicial homotopy between morphisms of simplicial sets $\varphi,\psi:S\to T$ is a collection of maps $h_i:S_n\to T_{n+1}$ for $0\leq i\leq n$ satisfying 
$\partial_0h_0=\varphi_n,$ $\partial_{n+1}h_n=\psi_n$ and
\begin{equation}
\partial_i h_j = 
\begin{cases}
    h_{j-1} \partial_i, & i<j,\\
    \partial_i h_{i-1}, & i=j\neq 0,\\
    h_j \partial_{i-1}, & i>j+1.
\end{cases}
\hspace{1cm}
s_ih_j = 
\begin{cases}
h_{j+1}s_i, & i\leq j,\\
h_js_{i-1}, & i>j
\end{cases}
\end{equation}
(see \cite[Def.5.1]{may1992simplicial}). Note that the last set of relations can be rewritten as
\begin{equation}\label{eq:simp:homopy:eq}
h_j s_i = 
\begin{cases}
s_i h_{j-1} , & i< j,\\
s_{i+1} h_j, & i \geq j.
\end{cases}    
\end{equation}
If $\varphi=\psi,$ then the constant homotopy from $\varphi$ to itself is defined by $h_i=s_i\varphi.$ 

Simplicial homotopies between $\varphi$ and $\psi$ are in bijection with maps $\tilde h:S\times \Delta^1 \to T$ such that $\varphi =\tilde hi$ and $\psi=\tilde hi',$ where $i,i':S\to S\times \Delta^1$ are the maps induced by the  maps $\Delta^0\to \Delta^1.$ Then the constant homotopy corresponds to the composition with the projection $S\times \Delta^1 \to S \to T.$

For any two maps between sets  $f,g:X\to Y,$ we consider a collection of maps $h^{f,g}_i:(\SS(\Delta(X))_n \to (\SS(\Delta(Y))_{n+1},$ for $0\leq i\leq n,$ defined by the formula 
\begin{equation}\label{eq:homotopy_h^{f,g}}
h^{f,g}_i(x_0,\dots,x_n) = (f(x_0),\dots,f(x_{i}),g(x_i), \dots,g(x_n)).
\end{equation}
It is easy to see that $h^{f,g}=(h^{f,g}_i)$ defines a simplicial homotopy between $\SS(\Delta(f))$ and $\SS(\Delta(g)).$ Moreover,  $h^{f,f}$ is the constant homotopy from $f$ to $f.$

\begin{proposition}\label{prop:c-short:homotopy}
Let $n\geq 0,$ and $f,g:X\to Y$ be a $L$-locally $(a,b)$-Lipschitz pair.  
Then the simplicial homotopy $h^{f,g}$ restricts to a simplicial homotopy between the following induced maps
\begin{equation}
f_*\sim g_* : \VR^{\nu,n}_LX \longrightarrow \VR^{\nu,n+1}_{a\cdot L+b\cdot C_{n+1}}Y.
\end{equation}
\end{proposition}
\begin{proof} Set $h_i=h_i^{f,g}.$ 
If  $x,x'\in X$ and  $d(x,x')\in L,$ all three numbers $d(f(x),f(x')),$ $d(f(x),g(x')),$ $d(g(x),g(x'))$ are less or equal than $a\cdot d(x,x')+b.$ Hence for any $\sigma=(x_0,\dots,x_m)\in (\VR_L^\nu X)_m$ the sequences $s_i\sigma$ and $h_i\sigma$ satisfy the assumption of Lemma \ref{lemma:main_inequality}. Therefore, using that $\WW(s_i \sigma )=\WW(\sigma),$ we obtain  
\begin{equation} \label{eq:lemma:h_i}
\WW(h_i \sigma)\leq a\cdot \WW(\sigma)+b\cdot C_{m+1}.
\end{equation}

Let us prove that for $\sigma\in (\VR^{\nu,n}_L X)_m$ and any $m,$ we have $h_i\sigma \in (\VR^{\nu,n+1}_{a\cdot L+b\cdot C_{n+1}} Y)_{m+1}.$  
It is sufficient to prove that for any natural $m,$ and any $\sigma\in (\VR^{\nu,n}_\RR X)_m$ we have $h_i\sigma\in (\VR^{\nu,n+1}_\RR Y)_{m+1}$ and 
\begin{equation}
\WW(h_i \sigma )\leq a\cdot \WW(\sigma)+b\cdot C_{n+1}.
\end{equation}
If $m\leq n,$ it follows from \eqref{eq:lemma:h_i}.
Now assume that $m\geq n+1$ and prove the statement by induction. Since $m\geq n+1,$ any $m$-simplex of the $n$-skeleton $\sigma\in (\VR^{\nu,n}_LX)_m$ can be obtained 
from an $(m-1)$-simplex $\sigma'\in  (\VR^{\nu,n}_LX)_{m-1}$ by applying a degeneracy map $\sigma=s_j(\sigma')$. Using \eqref{eq:simp:homopy:eq}, we obtain that we have an identity of the form $h_is_j = s_{j'}h_{i'}.$ 
The inductive assumption says that $h_{i'}\sigma'  \in (\VR^{\nu,n+1}_\RR Y)_{m}.$ Therefore, we obtain $h_i \sigma = s_{j'}h_{i'} \sigma' \in (\VR^{\nu,n+1}_\RR Y)_{m+1}.$ Moreover, we have 
\begin{equation}
\begin{split}
\WW(h_i\sigma) &= 
\WW(h_is_j\sigma') = \WW(s_{j'} h_{i'} \sigma') = \WW(h_{i'}\sigma') \\ 
&  \leq a\cdot \WW(\sigma') + b\cdot C_{n+1} = a\cdot \WW(\sigma)+b\cdot C_{n+1},
\end{split}
\end{equation}
which finishes the proof. 
\end{proof}

\section{Stability theorem}

In this section, we will prove a stability theorem for homology of $\nu$-Vietoris-Rips simplicial sets. We will show that the interleaving distance between the persistent modules  $H_n(\VR^{\nu}_{<*}X)$ and $H_n(\VR^{\nu}_{<*}Y)$ is bounded above by $2C_{n+2}(\nu)\cdot d_{GH}(X,Y),$ where $d_{GH}$ is the Gromov-Hausdorff distance. Throughout this section we denote by $X$ and $Y$ some metric spaces, by $\nu$ a distance matrix norm, and set $C_n=C_n(\nu).$

Let us recall the definitions of the Hausdorff distance and the Gromov-Hausdorff distance. If $Z$ is a metric space,  the Hausdorff distance $d_H(X,Y)$ between subsets  $X,Y\subseteq Z$ is the infimum of constants $c > 0$ such that for any $x\in X$ there exists $y\in Y$ such that $d(x,y)\leq c,$ and for any $y\in Y$ there exists $x\in X$ such that $d(y,x)\leq c.$ 
The Gromov-Hausdorff distance $d_{GH}(X,Y)$ between any metric spaces $X$ and $Y$ is defined as the infimum of $d_H(i(X),j(Y))$ taken by all isometric embeddings of $X$ and $Y$ to some metric space $i: X \hookrightarrow Z \hookleftarrow Y:j.$ 
Here we do not assume that $X$ and $Y$ are compact, so the distance $d_{GH}(X,Y)$ can be zero for non-isometric spaces, and can also be infinite. 

\begin{lemma}\label{lemma:distance_equivalent}
Let $\delta > 2 d_{GH}(X,Y).$ Then there exist $\delta$-short maps $f:X\to Y$ and $g:Y\to X$ such that $({\sf id}_X,gf)$ and $({\sf id}_Y,fg)$ are $2\delta$-short pairs. 
\end{lemma}
\begin{proof}
By the definition of the Gromov-Hausdorff distance, we can choose some embeddings to an ambient space $X,Y\hookrightarrow Z$ such that $\delta> 2d_H(X,Y).$ Then for any $x\in X,$ there exists an element $y\in Y$ such that $d(x,y)\leq \delta/2.$ Choosing one $y=f(x)$ for each $x,$ we define a function $f:X\to Y$ such that 
$d(x,f(x))\leq \delta/2.$
Similarly, we define a function $g:Y\to X$ such that $d(y,g(y))\leq \delta/2$ for any $y\in Y.$ These inequalities imply the following inequalities 
\begin{equation}
d( f(x),f(x'))\leq d(x,x') + \delta, \hspace{1cm} d( g(y),g(y'))\leq d(y,y') + \delta.   
\end{equation}
\begin{equation} 
d(x, gf(x')) \leq d(x,x') + \delta, \hspace{1cm} d(y, fg(y')) \leq d(y,y') + \delta.
\end{equation}
for any $x,x'\in X$ and $y,y'\in Y.$ 
Since $f$ and $g$ are $\delta$-short, their compositions $gf$ and $fg$ are $2\delta$-short. The statement follows. 
\end{proof}

Let us recall the definition of the interleaving distance between persistent objects \cite{bubenik2015metrics}, \cite{oudot2017persistence}. By a persistent object in a category $\CC,$ we mean a functor $P:\RR\to \CC,$ where $\RR$ is treated as a posetal category. A morphism of persistent objects is a natural transformation. 
For a real number $c,$ we denote by $P[c]$ the shifted persistent object, which is defined as the composition of $P$ with the functor $+c:\RR\to \RR.$ 
Then we have a natural transformation $\theta^c : P\to P[c]$ defined by $\theta^c_a = P(a,a+c):P(a)\to P(a+c).$  For $c>0,$ two persistent objects $P,P'$ are called $c$-interleaved, if there are morphisms $\varphi:P\to Q[c]$ and $\psi:Q\to P[c]$ such that the following diagrams are commutative.
\begin{equation}
\begin{tikzcd}
P\ar{r}{\varphi} \ar{rd}[below left]{\theta^{2c}} & Q[c] \ar{d}{\psi[c]} \\
  & P[2c]
\end{tikzcd}
\hspace{1cm} 
\begin{tikzcd}
Q\ar{r}{\psi} \ar{rd}[below left]{\theta^{2c}} & P[c] \ar{d}{\varphi[c]} \\
  & Q[2c]
\end{tikzcd}
\end{equation}
Then the interleaving distance between $P$ and $Q$ is defined as the infimum of the numbers $c>0$ such that $P$ and $Q$ are $c$-interleaved 
\begin{equation}
d_{\sf int}(P,Q) = \inf \{c>0\mid P,Q \text{ are $c$-interleaved}\}. 
\end{equation}

For a metric space $X$ and a commutative ring $\KK,$ we consider the persistent module 
\begin{equation}
H_n(\VR^{\nu}_{<*}X) : \RR \longrightarrow {\sf Mod}(\KK), \hspace{1cm} r\mapsto H_n(\VR^{\nu}_{<r} X), 
\end{equation}
where $H_n(-)=H_n(-;\KK).$

\begin{theorem}\label{theorem:stability} For any  metric spaces $X$ and $Y,$ any distance matrix norm $\nu,$ any commutative ring $\KK$ and any $n\geq 0$ we have 
\begin{equation}
d_{\sf int}( H_n(\VR_{<*}^{\nu} X), H_n(\VR_{<*}^{\nu} Y)) \leq 2 C_{n+2}(\nu) 
 \cdot   d_{GH}( X,Y ).
\end{equation}
The same inequality holds for the non-strict versions of these persistent modules $H_n(\VR^{\nu}_{\leq *}X)$ and $H_n(\VR^{\nu}_{\leq *}Y).$ 
\end{theorem}
\begin{proof}
Fix some $\delta >2d_{GH}(X,Y).$ By Lemma \ref{lemma:distance_equivalent}, there exist $\delta$-short maps $f:X\to Y$ and $g:Y\to X$ such that $({\sf id}_X,gf)$ and $({\sf id}_Y,fg)$ are $2\delta$-short pairs.  
Then Proposition \ref{prop:c-short:induces} and inequality $C_{n+1}\leq C_{n+2}$ imply that  $f$ and $g$ induce morphisms
\begin{equation}
   f_* : \VR^{\nu,n+1}_{<*} X \to  \VR^{\nu,n+1}_{<*+\delta C_{n+2}} Y, \hspace{5mm}  g_* : \VR^{\nu,n+1}_{<*} Y \to  \VR^{\nu,n+1}_{<*+\delta C_{n+2}} X. 
\end{equation}
Proposition  \ref{prop:c-short:homotopy} implies that the diagrams 
\begin{equation}
\begin{tikzcd}[column sep=5mm]
\VR^{\nu,n+1}_{<*} X \ar{r}{f_*} \ar{rd}[below left]{\theta} & \VR^{\nu,n+1}_{<*+\delta C_{n+2}} Y \ar{d}{g_*} \\
& \VR^{\nu,n+2}_{<*+2\delta C_{n+2}} X
\end{tikzcd}
\hspace{5mm}
\begin{tikzcd}[column sep=5mm]
\VR^{\nu,n+1}_{<*} Y \ar{r}{g_*} \ar{rd}[below left]{\theta} & \VR^{\nu,n+1}_{<*+\delta C_{n+2}} X \ar{d}{f_*} \\
& \VR^{\nu,n+2}_{<*+2\delta C_{n+2}} Y
\end{tikzcd}
\end{equation}
are commutative up to homotopy. Since  $H_n(\VR^\nu_{<r}X) = H_n(\VR^{\nu,m}_{<r}X)$ for $m\geq n+1,$ we obtain that the persistent modules $H_n(\VR^{\nu}_{<*} X)$ and $H_n(\VR^{\nu}_{<*} Y)$ are $\delta C_{n+2}$-interleaved. 
\end{proof}

\begin{corollary}\label{cor:stability}
Under the assumption of Theorem \ref{theorem:stability} and for any $p\in [1,\infty],$ there are inequalities 
\begin{equation}
d_{\sf int}( H_n(\VR_{<*}^{p} X), H_n(\VR_{<*}^{p} Y)) \leq 2 
(n+2)^{\frac{1}{p}}\cdot  d_{GH}( X,Y ),
\end{equation}
and 
\begin{equation}
d_{\sf int}( H_n(\vr_{<*}^{p} X), H_n(\vr_{<*}^{p} Y)) \leq 2 
(n+2)^{\frac{1}{p}}\cdot  d_{GH}( X,Y ),
\end{equation}
and 
\begin{equation}
d_{\sf int}( \MH_{n,< *}(X), \MH_{n,< *} ( Y) ) \leq 2 
(n+2)\cdot  d_{GH}( X,Y ).
\end{equation}
The same inequalities hold for non-strict versions of these persistent modules $H_n(\VR^p_{\leq *} -),$ $H_n(\vr^p_{\leq *}-)$ and  $\MH_{n,\leq *}(-).$  
\end{corollary}

\section{Embeddings inducing weak equivalences} 

The main goal of this section is to introduce  some conditions for an embedding of metric spaces  $A\hookrightarrow X$ under which the induced embedding of simplicial sets  $\VR^\nu_{<r} A \hookrightarrow  \VR^\nu_{<r} X$ is a weak equivalence (Theorem \ref{th:embedding}). 
As a corollary, we obtain that the canonical embedding of a metric space to its metric completion $X\hookrightarrow \hat X$ induces a weak equivalence $\VR^\nu_{<r} X \overset{\sim}\to \VR^\nu_{<r} \hat X$ (Corollary \ref{cor:completion}), and that for a convex subset $X\subseteq \RR^n$ the simplicial set $\VR^\nu_{<r} X$ is weakly contractible (Corollary \ref{cor:convex}).

In this section we denote by $\nu$ a distance matrix norm, by $r>0$ a positive number, and we assume that the left infinite interval that we considered in previous sections is open $L=(-\infty, r),$ because the results of this section do not hold for closed intervals (Remarks \ref{remark:non-iso-for-closed} and \ref{remark:shortly_contr}).

\begin{lemma}
\label{lemma:finite_space} 
For a finite metric space $X$ and any $n\geq 0,$ there exists $\varepsilon>0$ such that the map
\begin{equation}
\VR^{\nu,n}_{<r-\varepsilon} X \longrightarrow 
\VR^{\nu,n}_{<r} X
\end{equation}
is an isomorphism. 
\end{lemma}
\begin{proof}
The set of all possible weights $\WW(x_0,\dots,x_k),$ for $k\leq n,$ is a finite subset of $\RR.$ Then there exists $\varepsilon$ such that $\WW(x_0,\dots,x_k)<r$ implies $\WW(x_0,\dots,x_k)<r-\varepsilon$ for $k\leq n.$
\end{proof}

An $r$-locally $c$-short map is an $L$-locally $c$-short map for $L=(-\infty,r).$ More precisely, a map $f:X\to Y$ is $r$-locally $c$-short, if $d(x,x')<r$ implies   $d(f(x),f(x'))\leq d(x,x')+c.$ A pair of maps $f,g:X\to Y$ is called an $r$-locally $c$-short pair, if $d(x,x')<r$ implies $d(f(x),g(x'))\leq d(x,x')+c.$ If $r=\infty,$ then we use terms ``$c$-short map'', and ``$c$-short pair of maps''.

\begin{definition}[$r$-locally $c$-short equivalence]
We say that two $r$-locally $c$-short maps $f,g:X\to Y$ are $r$-locally $c$-shortly equivalent, if there is a sequence of $r$-locally $c$-short maps $f=f_0,\dots,f_n=g$ such that $(f_i,f_{i+1})$ is a $r$-locally $c$-short pair for any $i.$
\end{definition}

\begin{definition}[$r$-locally $c$-short equivalence for maps of pairs of metric spaces] A pair of metric spaces is a pair $(X,A),$ where $X$ is a metric space and $A$ is its subspace. An $r$-locally $c$-short map of pairs $f:(X,A)\to (Y,B)$ is an $r$-locally $c$-short map $f:X\to Y$ such that $f(A)\subseteq B.$ Two $r$-locally $c$-short maps of pairs $f,g:(X,A)\to (Y,B)$ are $r$-locally $c$-shortly equivalent relative to $A$, if there is a sequence of $r$-locally $c$-short maps of pairs $f=f_0,\dots,f_n=g:(X,A)\to (Y,B)$ such that $(f_i,f_{i+1})$ is an $r$-locally $c$-short pair of maps from $X$ to $Y$ and $f_i(a)=f_{i+1}(a)$ for any $i$ and $a\in A.$
\end{definition}

\begin{theorem}[{cf. \cite[(2.2)]{hausmann1994vietoris}}] 
\label{th:embedding}
Let $X$ be a metric space and $A\subseteq X$ be its subspace, with inclusion denoted by  $\iota:A\hookrightarrow X.$ Assume that for any $\varepsilon>0$ there is an $r$-locally $\varepsilon$-short map $\rho_\varepsilon : X\to A$ such that $\rho_\varepsilon\iota={\sf id}_A$ and the maps of pairs $\iota \rho_\varepsilon, {\sf id}_X :(X,A) \to (X,A)$ are $r$-locally $\varepsilon$-shortly equivalent. Then $\iota$ induces a weak equivalence 
\begin{equation}
\VR^\nu_{<r} A \overset{\sim}\longrightarrow  \VR^\nu_{<r} X.
\end{equation}
\end{theorem}
Before we prove the theorem, we need to prove the following version of the Whitehead theorem, which seems to be well known but we could not find a proper reference. If $S,T$ are simplicial sets, and $S'\subseteq S$ is a simplicial subset, we say that two maps $f,g:S \to T$ are homotopy equivalent  relative to $S',$ if there is a sequence of maps  $f=f_0,\dots,f_n=g:S \to T$ such that, for any $0\leq i<n,$ there exists a simplicial homotopy $S\times \Delta^1 \to T$ connecting either $f_i$ with $f_{i+1}$ or $f_{i+1}$ with $f_i$ whose restriction on $S'\times \Delta^1$ is a constant simplicial homotopy. In particular, we assume that $f|_{S'} = g|_{S'}.$ 

\begin{lemma}\label{lemma:anodyne} Let $f:T'\to T$ be a morphism of simplicial sets. Assume that, for any inclusion of simplicial sets having a finite number of non-degenerate simplices $S'\hookrightarrow S$ and for any commutative square, 
\begin{equation}
\begin{tikzcd}
S' \ar[r] \ar[d,hookrightarrow] & T' \ar[d,"f"] \\
S \ar[r] \ar[ru,dashed,"h"] & T
\end{tikzcd}
\end{equation}
there exists a map $h:S\to T'$ such that the upper triangle is commutative and the lower triangle is commutative up to homotopy relative to $S'.$
Then the map $f:T\to T'$ is a weak equivalence. 
\end{lemma}
\begin{proof} Let's apply the Kan fibrant replacement ${\sf Ex}^\infty$  to the map  $f:T \to T'$ (see \cite[\S III.4]{goerss2009simplicial}). 
Then we need to show that ${\sf Ex}^\infty f: {\sf Ex}^\infty\:T' \to  {\sf Ex}^\infty\:T$ is a weak equivalence. By the simplicial Whitehead theorem \cite[Th.7.2]{kan1957css}, we need to show that for any commutative square
\begin{equation}\label{eq:wh}
\begin{tikzcd}
\partial \Delta^n \ar[r,"a"] \ar[d,hookrightarrow] &  {\sf Ex}^\infty\:T' \ar[d] \\
\Delta^n  \ar[r,"b"] \ar[ru,dashed,"h"]   & {\sf Ex}^\infty\:T
\end{tikzcd}
\end{equation}
there exists a map $h:\Delta^n \to {\sf Ex}^\infty T'$ such that the upper triangle is commutative and the lower triangle is commutative up to homotopy relative to $\partial \Delta^n$. Since $\partial \Delta^n$ and $\Delta^n$ have a finite number of simplices, the maps to ${\sf Ex}^\infty$ factor through ${\sf Ex}^N$ for sufficiently large $N.$ Therefore, we obtain a commutative diagram 
\begin{equation}
\begin{tikzcd}
\partial \Delta^n \ar[r,"a'"] \ar[d,hookrightarrow] &  {\sf Ex}^N\:T' \ar[d] \\
\Delta^n  \ar[r,"b'"]  & {\sf Ex}^N\:T
\end{tikzcd}
\end{equation}
Using the adjunction ${\sf sd} \dashv {\sf Ex},$ we obtain the following commutative square. 
\begin{equation}
\begin{tikzcd}
{\sf sd}^N \partial \Delta^n \ar[r,"a''"] \ar[d,hookrightarrow] &  T' \ar[d] \\
{\sf sd}^N \Delta^n  \ar[r,"b''"] \ar[ru,dashed,"h''"]  & T
\end{tikzcd}
\end{equation}
By the assumption, there is a map $h'':{\sf sd}^N \Delta^N \to T'$ such that the upper triangle is commutative and the lower triangle is commutative up to homotopy relative to ${\sf sd}^N\partial \Delta^n$. Then the map $h$ can be constructed as a composition of the adjoint map $h':\Delta^n \to {\sf Ex}^N T'$ and the monomorphism  $ {\sf Ex}^N T' \mono {\sf Ex}^\infty T'.$

 In order to prove that the lower triangle of the diagram \eqref{eq:wh} is commutative up to homotopy relative to $\partial \Delta^n$, we need to observe that for any homotopy $({\sf sd}^N \Delta^n)\times \Delta^1 \to T'$, we obtain a map ${\sf sd}^N(\Delta^n\times \Delta^1)\to T'$ by precomposing it with
\begin{equation}
    {\sf sd}^N(\Delta^n \times \Delta^1)\to ({\sf sd}^N\Delta^n)\times ({\sf sd}^N\Delta^1) \xrightarrow{1\times\lambda} ({\sf sd}^N\:\Delta^n)\times \Delta^1,
\end{equation}
where $\lambda:{\sf sd}^N\:\Delta^1\to \Delta^1$ is an iteration of the last vertex map. Moreover, the composition of maps 
\begin{equation}
    {\sf sd}^N\Delta^n \cong {\sf sd}^N(\Delta^n\times \Delta^0)\xrightarrow{{\sf sd}^N(1\times \partial^i)} {\sf sd}^N(\Delta^n\times \Delta^1)\to T'
\end{equation}
equals to
\begin{equation}
    {\sf sd}^N\Delta^n\cong{\sf sd}^N\Delta^n\times\Delta^0\xrightarrow{1\times \partial^i}({\sf sd}^N\:\Delta^n)\times \Delta^1\to T',
\end{equation}
for $i\in\{0,1\} $, and we have the following commutative diagram
\begin{equation}
    \begin{tikzcd}
        {\sf sd}^N(\partial\Delta^n\times\Delta^1) \ar[r,] \ar[d,"{\sf sd}^N(\iota\times 1)"] & {\sf sd}^N\partial\Delta^n\times {\sf sd}^N\Delta^1\ar[r,"1\times\lambda"] \ar[d,"{\sf sd}^N(\iota)\times 1"] & {\sf sd}^N\partial\Delta^n\times\Delta^1 \ar[d,"{\sf sd}^N(\iota)\times1"]\\
        {\sf sd}^N(\Delta^n\times\Delta^1) \ar[r] &  {\sf sd}^N\Delta^n\times{\sf sd}^N\Delta^1 \ar[r,"1\times\lambda"] & {\sf sd}^N\Delta^n\times\Delta^1 
    \end{tikzcd}
\end{equation}
where $\iota:\partial\Delta^n\hookrightarrow\Delta^n$ is the inclusion. Therefore, by our observation and the adjunction ${\sf sd}\dashv {\sf Ex}$, the sequence of homotopies $({\sf sd}^N\Delta^n) \times \Delta^1 \to T'$ defines a sequence of homotopies $\Delta^n\times \Delta^1 \to {\sf Ex}^N T'$ such that the lower triangle of the diagram \eqref{eq:wh} is commutative up to homotopy relative to $\partial\Delta^n$, as desired.
\end{proof}

\begin{proof}[Proof of Theorem \ref{th:embedding}] 
By Lemma \ref{lemma:anodyne}, it is sufficient to prove that for any pair of simplicial sets $(S,S')$ with finitely many non-degenerate simplices any map of pairs of simplicial sets  $f:(S,S')\to (\VR^{\nu}_{<r} X,\VR^{\nu}_{<r} A)$ is 
homotopy equivalent relative to $S'$ to a map sending $S$ to $\VR^{\nu}_{<r} A.$ Set $X'=f_0(S_0)\subseteq X$ and $A'=f_0(S'_0)\subseteq A$ and choose $n$ such that $S=\sk_n(S)$ and $S'=\sk_n(S').$ Then the map $f$ factors through the inclusion  $(\VR^{\nu,n}_{<r} X',\VR^{\nu,n}_{<r} A') \to (\VR^{\nu}_{<r} X,\VR^{\nu}_{<r} A),$ and it is sufficient to prove the statement for 
\begin{equation}
 (S,S')=(\VR^{\nu,n}_{<r} X',\VR^{\nu,n}_{<r} A')   
\end{equation}
and $f$ induced by the embedding $\alpha :X'\hookrightarrow X.$ 

By Lemma \ref{lemma:finite_space}, we can choose $\varepsilon>0$ such that the map 
\begin{equation}
(\VR^{\nu,n}_{<r - \varepsilon C_{n+2}} X',\VR^{\nu,n}_{<r - \varepsilon C_{n+2} } A') \overset{\cong}\longrightarrow (\VR^{\nu,n}_{<r} X',\VR^{\nu,n}_{<r} A')    
\end{equation}
is an isomorphism. Proposition \ref{prop:c-short:induces} implies that the $r$-locally short maps $\alpha$ and $\iota \rho_\varepsilon \alpha$ induce maps
\begin{equation}
\alpha_* , (\iota \rho_\varepsilon \alpha)_* : 
(\VR^{\nu,n}_{<r-\varepsilon C_{n+2} } X',\VR^{\nu,n}_{<r-\varepsilon 
 C_{n+2} } A') \longrightarrow (\VR^{\nu}_{<r} X,\VR^{\nu}_{<r} A).  
\end{equation} 
Then it is enough to prove that these maps are  homotopic $\alpha_* \sim (\iota \rho_\varepsilon \alpha)_*$ relative to $\VR^{\nu,n}_{<r-\varepsilon 
 C_{n+2} } A'.$  Since $\iota \rho_\varepsilon$ and ${\sf id}$ are $r$-locally $\varepsilon$-shortly equivalent morphisms of pairs $(X,A)\to (X,A),$ we have a sequence ${\sf id}=f_0,\dots, f_m = \iota \rho_\varepsilon :(X,A)\to (X,A)$ such that $(f_i,f_{i+1})$ is an $r$-locally $\varepsilon$-short pair  and $f_i|_A={\sf id}_A$. 
If we compose them with $\alpha,$ we obtain that  $(f_i\alpha ,f_{i+1}\alpha)$  are $r$-locally $\varepsilon$-short   pairs so that $\iota \rho_\varepsilon\alpha$ is $r$-locally $\varepsilon$-shortly equivalent to $\alpha$ as maps of pairs $(X',A')\to (X,A)$. By Proposition \ref{prop:c-short:homotopy} and the formula \eqref{eq:homotopy_h^{f,g}}, we obtain that there is a simplicial homotopy 
\begin{equation}
(f_i\alpha )_* \sim (f_{i+1} \alpha)_* : 
(\VR^{\nu,n}_{<r-\varepsilon C_{n+2}} X',\VR^{\nu,n}_{<r-\varepsilon C_{n+2} } A') \longrightarrow (\VR^{\nu}_{<r} X,\VR^{\nu}_{<r} A),  
\end{equation} 
relative to $\VR^{\nu,n}_{<r-\varepsilon C_{n+2} } A'$ which finishes the proof. 
\end{proof}

 \begin{proposition}[{cf. \cite[(2.4)]{hausmann1994vietoris}}]
Let $X$ be a metric space and $A\subseteq X$ be its dense subspace.  Then the inclusion 
$\VR^\nu_{<r} A \to \VR^\nu_{<r} X$
is a weak equivalence. 
\end{proposition}
\begin{proof}
Take  $\varepsilon>0.$ Since $A$ is dense in $X,$ for any $x\in X$ there exists $a\in A$ such that $d(x,a)\leq \varepsilon/2.$ Hence we can choose a map $\rho_\varepsilon : X\to A$ such that $d(x,\rho_\varepsilon(x))\leq \varepsilon/2$ and $\rho_\varepsilon(a)=a$ for $a\in A.$ Then we obtain $d(x,\rho_\varepsilon(x'))\leq d(x,x')+\varepsilon/2$ and $d(\rho_\varepsilon(x),\rho_\varepsilon(x')) \leq d(x,x')+\varepsilon.$ Therefore $\rho_\varepsilon$ is $\varepsilon$-short and $({\sf id},\iota \rho_\varepsilon)$ is a $\varepsilon$-short pair of maps $(X,A)\to (X,A)$. Then the statement follows from Theorem \ref{th:embedding}. 
\end{proof}

\begin{corollary}[{cf. \cite[(2.5)]{hausmann1994vietoris}}] 
\label{cor:completion}
Let $\hat X$ be the metric completion of a metric space $X.$ Then the map
$\VR^\nu_{<r} X \overset{\sim}\to  \VR^\nu_{<r} \hat X$
is a weak equivalence for any $r>0.$ In particular, we have an isomorphism for strictly blurred magnitude homology 
\begin{equation}\label{eq:open_blured_completion}
\MH_{n,<r}(X)\cong \MH_{n,<r}(\hat X).     
\end{equation} 
\end{corollary}

\begin{remark}\label{remark:non-iso-for-closed} The analogue of the isomorphism \eqref{eq:open_blured_completion} does not hold for the non-strictly blurred magnitude homology 
\begin{equation}
\MH_{n,\leq r}(X)\not\cong \MH_{n,\leq r}(\hat X).
\end{equation}
Indeed, if we had such an isomorphism, the long exact sequence \eqref{eq:long_exact_magnitude} would imply the isomorphism for magnitude homology. However, it is easy to construct an example of a metric space $X$ such that 
\begin{equation}
\MH_{1,1}(X) \not\cong \MH_{1,1}(\hat X)
\end{equation}
(see the description of $\MH_{1,r}$ in \cite[Cor. 4.5]{leinster2021magnitude}).
For example,  we can take $X\subseteq \RR^2$ defined by 
\begin{equation}
X=\RR^2\setminus((0,1)\times \{0\}), \end{equation}
where $(0,1)\subseteq \RR$ is the open interval. Then we have $\hat X=\RR^2,$  $\MH_{1,1}(X)\neq 0$ and $\MH_{1,1}(\hat X)=0.$ This also implies that Theorem \ref{th:embedding} does not hold for the non-strict version of the simplicial set $\VR^\nu_{\leq r}.$
\end{remark}

\begin{definition}[$r$-locally short deformation retract]\label{def:shortly_contractible} A metric subspace $A\subseteq X$ is called $r$-locally short deformation retract, if there is a short map $\rho:X\to A$ such that $\rho(a)=a$ for  $a\in A$ and there is a map $h:X\times  [0,c] \to X,$ for some $c>0,$ such that $h(x,0)=\rho(x),$  $h(x,c)=x$ for $x\in X,$ $h(a,t)=a$ for $a\in A,t\in [0,c]$ and 
\begin{equation}
d(x,x')<r \hspace{5mm} \Rightarrow \hspace{5mm} d(h(x,t),h(x',t'))\leq d(x,x')+|t-t'|.
\end{equation}
If $A$ consists of one point, then $X$ is called $r$-locally shortly contractible. If $r=\infty,$ then we use the terms ``short deformation retract'' and ``shortly contractible''.  
\end{definition}

\begin{proposition}\label{prop:short_def_retract}
For an $r$-locally short deformation retract $A\subseteq X$ the inclusion $\VR^\nu_{<r} A \to \VR^\nu_{<r} X$ is a weak equivalence.  
\end{proposition}
\begin{proof} In order to use Theorem \ref{th:embedding}, we need to show that for any $\varepsilon>0$ the short maps of pairs ${\sf id}, \iota \rho : (X,A)\to (X,A)$ are $r$-locally $\varepsilon$-shortly equivalent. Consider $n$ such that $1/n<\varepsilon$ and set $f_i(x)=h(x,i/n)$ for $0\leq i\leq \lceil cn \rceil$ assuming that $h(x,t)=h(x,c)$ for $t>c.$  Then, if $d(x,x')<r,$ we have
\begin{equation}
\begin{split}
d(f_i(x),f_{i+1}(x'))&=d(h(x,i/n),h(x',(i+1)/n))\\
& \leq d(x,x')+ 1/n \\ 
&\leq d(x,x')+ \varepsilon.
\end{split}
\end{equation}
Therefore the pair $(f_i,f_{i+1})$ is $r$-locally $\varepsilon$-short, $f_i(a)=a$ for $a\in A$, and hence ${\sf id}$ and $ \iota \rho$ are $\varepsilon$-equivalent maps of pairs $(X,A)\to (X,A)$. 
\end{proof}

\begin{corollary}\label{cor:short_contr}
If $X$ is shortly contractible, then $\VR^\nu_{<r} X$ is weakly contractible for any $\nu$ and $r.$
\end{corollary}

A subset $X\subseteq \RR^n$ is called \emph{star domain}, if there exists a point $x_0\in X$ such that for any $x\in X$ the line segment from $x_0$ to $x$ lies in $X.$ For example, convex sets are star domains.   

\begin{proposition}\label{prop:star:domain} For a star domain $X\subseteq \RR^n,$  the simplicial set $\VR^\nu_{<r} X$ is weakly contractible for any  $r>0.$  
\end{proposition}
\begin{proof} For simplicity we make a shift of $X$ such that $x_0=0.$ First we assume that $X$ is bounded and denote by $c>0$ a constant such that $\|x\|\leq c$ for any $x\in X.$
Consider the map $h:X\times_1 [0,c] \to X$ defined by $h(x,t)= (t/c) x.$ The map is short:   
\begin{equation}
\begin{split}
 \|(t/c)x - (t'/c)x'\| &\leq 
\|(t/c) x - (t/c) x'\|+\|(t/c)x' - (t'/c) x'\| \\
& \leq (t/c)\cdot \|x - x'\| + |(t-t')/c|\cdot \|x'\|   \\
&\leq \|x - x'\| + |t-t'|. 
\end{split}  
\end{equation}
Then $X$ is shortly contractible and we can use  Corollary \ref{cor:short_contr}. 
Now assume that $X$ is not bounded. Then we can present it as the union $X=\bigcup_n X_n,$ where $X_n$ is the intersection with the ball  $X\cap B_n(0)$ of radius $n.$ Then $\VR^\nu_{<r} X = \bigcup \VR^\nu_{<r} X_n$ is the union of the increasing sequence of simplicial subsets. Hence the union is homotopy equivalent to the homotopy colimit $\VR^\nu_{<r} X  \simeq  \hocolim \VR^\nu_{<r} X_n.$ Therefore, the fact that $\VR^\nu_{<r} X_n$ is weakly contractible for each $n$ implies that $\VR^\nu_{<r} X$ is weakly contractible. 
\end{proof}

\begin{corollary}\label{cor:convex}
For a convex subset $X\subseteq \RR^n,$  the simplicial set $\VR^\nu_{<r} X$ is weakly contractible for any  $r>0.$
\end{corollary}

\begin{corollary}
For a convex subset $X\subseteq \RR^n,$ both strictly and non-strictly blurred magnitude homology vanish 
\begin{equation}
\MH_{n,<r}(X)=\MH_{n,\leq r}(X)=0
\end{equation}
for $n\geq 1.$
\end{corollary}
\begin{proof}
For the strict case it follows from Corollary \ref{cor:convex}. For the non-strict case it follows from the long exact sequence \eqref{eq:long_exact_magnitude} and the fact that magnitude homology is trivial for convex sets \cite{kaneta2021magnitude}.  
\end{proof}

\begin{remark}
\label{remark:shortly_contr} Here we want to emphasize the difference between strictly and non-strictly blurred magnitude homology once again.  
Consider a star domain $X\subseteq \RR^2$ defined by 
\begin{equation}
 X=([0,1]\times \{0\})\cup (\{0\}\times [0,1]).   
\end{equation}
Then by Proposition \ref{prop:star:domain} we have $\MH_{n,<r}(X)=0$ for $n\geq 1.$ However, it is easy to see that $\MH_{1,r}(X)\neq 0$ (see \cite[Cor. 4.5]{leinster2021magnitude}) for any $0<r\leq \sqrt{2},$ and hence, using the long exact sequence \eqref{eq:long_exact_magnitude}, we obtain that $\MH_{1,\leq r}(X)\neq 0.$ 
\end{remark}

\begin{example}
Consider an arc  $X=\{e^{ix} \mid 0\leq x\leq c \} \subseteq \mathbb{C},$ where $\pi \leq c <2\pi,$ 
with the Euclidean metric.    
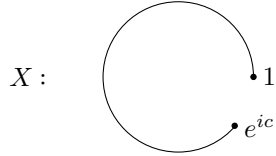
\[ 
\begin{tikzpicture}
\draw (0,0) arc (0:320:1);
\filldraw  (0,0) circle (1pt) node[right]{$1$};
\filldraw  (-0.25,-0.65) circle (1pt) node[right]{$e^{ic}$};
\node at (-3,0) {$X:$};
\end{tikzpicture}
\]
We claim that for any $r<d(1,e^{ic})$ the arc $X$ is $r$-locally shortly contractible, and hence,  $\VR^\nu_{<r} X$ is weakly contractible. Let us prove this. Consider a map $h:X \times [0,c]\to X$ defined by
\begin{equation}
h(e^{ix},t)= e^{ix (t/c)}. 
\end{equation}
We need to check the property from the definition of the $r$-locally short deformation retract. First note that if $|x-x'|\leq \pi$ and $0\leq \alpha\leq 1,$ then
\begin{equation}
|e^{ix \alpha} - e^{ix'\alpha}| \leq  |e^{ix} - e^{ix'}| \leq |x-x'|.
\end{equation}
Assume that $0\leq x,x'\leq c$ and $d(e^{ix},e^{ix'})<r.$ Then $|x-x'|\leq \pi$ and we have 
\begin{equation}
\begin{split}
|e^{ix(t/c)}-e^{ix'(t'/c)}| & \leq |e^{ix(t/c)}-e^{ix'(t/c)}| + |e^{ix'(t/c)}-e^{ix'(t'/c)}|  \\
&\leq |e^{ix}-e^{ix'}| + |t-t'|.
\end{split}
\end{equation}
Therefore $X$ is $r$-locally shortly contractible. In particular, we obtain that the strictly blurred magnitude homology is trivial 
\begin{equation}
\MH_{n,<r}(X)=0, \hspace{1cm} n\geq 1,\  0<r\leq d(1,e^{ic}).
\end{equation}
While the long exact sequence \eqref{eq:long_exact_magnitude} and the description of the first magnitude homology \cite[Cor. 4.5]{leinster2021magnitude} imply that 
\begin{equation}
\MH_{1,\leq r}(X) \cong \MH_{1,r}(X)\cong \mathbb{Z}^{\oplus (2^{\aleph_0})}, \hspace{1cm} 0<r<2.
\end{equation}
\end{example}

\section{Nerve theorem and Mayer-Vietoris sequence}

This section is devoted to results that allow us to reduce the computation of the  homotopy type of $\VR^\nu_L X$ of a metric space $X$ covered by a family of subsets $X=\bigcup_{i\in I}U_i$ to the homotopy types of $\VR^\nu_L U_i.$ In particular, we prove a version of the nerve theorem for the $\nu$-Vietoris-Rips simplicial sets. 

Let $S$ be a set (resp. simplicial set) and $\UU=(U_i)_{i\in I}$ be a family of subsets (resp. simplicial subsets). For a subset $\sigma \subseteq I$ we set $U_\sigma = \bigcap_{i\in \sigma} U_i.$ We denote by $N_\UU$ the nerve complex of this family 
\begin{equation}
N_\UU = \{\emptyset \neq \sigma \underset{\scalebox{0.6}{\sf fin}}\subseteq I \mid U_\sigma \neq \emptyset\}.
\end{equation}

\begin{theorem}[Nerve theorem]\label{theorem:nerve_theorem}
Let $L$ be a left infinite interval, $X$ be a metric space and $\UU=(U_i)_{i\in I}$ be a family of subsets $U_i\subseteq X.$ Assume that $\VR^\nu_L X=\bigcup_{i\in I} \VR^\nu_L U_i$ and $\VR^\nu_L U_\sigma$ is weakly contractible for any $\sigma\in N_\UU.$ Then 
there is a homotopy equivalence  
\begin{equation}
|\VR^\nu_L X| \simeq |N_\UU|.
\end{equation}
\end{theorem}
\begin{proof}
It is easy to see that for $\sigma\in N_\UU$ we have $ \VR^\nu_L U_\sigma = \bigcap_{i\in \sigma}(\VR^\nu_L U_i).$ Then the statement follows from the nerve theorem for simplicial sets that can be found in Appendix (Theorem \ref{th:simplicial_nerve_theorem}).
\end{proof}

For a metric space $X,$ a left infinite interval $L$ and points $x,y\in X$ we consider the ``$\nu$-ellipse with foci $x,$ $y$''  
\begin{equation}
\Ell_L^\nu(x,y) = \{ a\in X \mid  \WW_\nu(x,a,y) \in L \}.
\end{equation}
Note that, since $d(x,a)=\WW_\nu(x,a)\leq \WW_\nu(x,a,y),$ we have 
\begin{equation}\label{eq:elli:B}
\Ell^\nu_L(x,y) \subseteq B_L(x),
\end{equation}
where $B_L(x)=\{a\in X\mid d(x,a)\in L\}.$

If $\nu=\nu_p,$ we denote the $\nu$-ellipse by $\Ell^p_L(x,y),$ and if $\nu=\nu^{\sf sym}_{p},$ then we denote it by $\Ell^{{\sf sym}\text{-}p}_L(x,y).$ We use the term ``ellipce'' because 
\begin{equation}
\Ell^1_L(x,y) = \{a\in X\mid d(x,a)+d(a,y)\in L\}.
\end{equation}
For $p=\infty$ we have 
\begin{equation}
\Ell^\infty_L(x,y) = B_L(x) \cap B_L(y).
\end{equation}
For the symmetric case and $p=1,$ we have
\begin{equation}
\Ell^{{\sf sym}\text{-}1}_L(x,y) = B_{L-d(x,y)}(x) \cup \Ell^1_{L}(x,y) \cup B_{L-d(x,y)}(y).
\end{equation}

\begin{lemma}\label{lemma:union:family}
Let $\UU=(U_i)_{i\in I}$ be a family of subsets of a metric space $X$ and $L$ be a left infinite interval. Assume that for any $x,y\in X$ such that $d(x,y)\in L$ there exists $i\in I$ such that $\Ell_L^\nu(x,y)\subseteq U_i.$ Then 
\begin{equation}
\VR^\nu_L X = \bigcup_{i\in I} \VR^\nu_L U_i.  
\end{equation}
\end{lemma}
\begin{proof}
Take $(x_0,\dots,x_n)\in (\VR^\nu_L X)_n.$ Then for any $1\leq k\leq n$ we have 
$(x_0,x_k,x_n)\in (\VR^\nu_L X)_2$. It follows that $x_k\in \Ell_L^\nu(x,y)(x_0,x_n)$ and $d(x_0,x_n)\in L.$ Hence, if $\Ell_L^\nu(x_0,x_n)\subseteq U_i$ then $(x_0,\dots,x_n)\in \VR^\nu_L U_i.$
\end{proof}

\begin{corollary}\label{cor:nerve_theorem}
Let $L$ be a left infinite interval, $X$ be a metric space and $\UU=(U_i)_{i\in I}$ be a family of subsets $U_i\subseteq X$ such that $\VR^\nu_L U_\sigma$ is weakly contractible for any $\sigma\in N_\UU.$ 
Assume that for any $x,y\in X$ such that $d(x,y)\in L$ there exists $i\in I$ such that $\Ell_L^\nu(x,y)\subseteq U_i.$ Then there is a homotopy equivalence  
\begin{equation}
|\VR^\nu_L X| \simeq |N_\UU|.
\end{equation}
\end{corollary}

\begin{proposition}[Mayer-Vietoris sequence]
Let $L$ be a left infinite interval, $X$ be a metric space and $U,V\subseteq X$ be subsets such that for any $x,y\in X$ such that $d(x,y)\in L$ we either have  $\Ell_L^\nu(x,y)\subseteq U,$ or $\Ell_L^\nu(x,y)\subseteq V.$ Then the following commutative square is a homotopy pushout.
\begin{equation}\label{eq:pushout:MV}
\begin{tikzcd}
    \VR^\nu_L(U\cap V) \ar{r} \ar{d} & \VR^\nu_L V \ar{d} \\
    \VR^\nu_L U \ar{r} & \VR^\nu_L X
\end{tikzcd}
\end{equation}
In particular, there is a long exact sequence 
\begin{equation}
\dots \to H_n(\VR^\nu_L(U\cap V)) \to H_n(\VR^\nu_L V) \oplus H_n(\VR^\nu_L U) \to H_n(\VR^\nu_L X) \to \dots, 
\end{equation}
where $H_n(-)=H_n(-,\KK)$ is homology with coefficients in a commutative ring $\KK.$
\end{proposition}
\begin{proof} By Lemma \ref{lemma:union:family} we have $\VR^\nu_L X = \VR^\nu_L U \cup \VR^\nu_L V.$ It is easy to see that $\VR^\nu_L(U\cap V) =(\VR^\nu_L U)\cap (\VR^\nu_L V).$ Then the square is a pushout consisting of embeddings. Therefore it is a homotopy pushout.
\end{proof}

\section{Geodesic spaces and Riemannian manifolds}

This section is devoted to the proof of the fact that for a geodesic space $X$ and a sufficiently small scale parameter $r$ the geometric realization of the simplicial set $\VR^\nu_{<r} X$ is homotopy equivalent to $X$ (Theorem \ref{th:manifold}).

Let $X$ be a metric space. A minimizing geodesic is a map $\gamma : I \to X$ from an interval $I\subseteq \RR$ such that 
$d(\gamma(s),\gamma(t))=|t-s|$ for any $s,t\in I$ (see \cite[Def.1.9]{gromov1999metric}). A metric space  $X$ is called geodesic, if for any $x,y\in X$ there exists a minimizing geodesic $\gamma_{x,y}:[0,d(x,y)]\to X$ from $x$ to $y.$ For example, any Riemannian manifold can be treated as a geodesic space \cite[\S 1.10, \S 1.14]{gromov1999metric}.

For a geodesic space $X,$ we define a real number $r(X)\geq 0$ as the supremum of numbers $r\geq 0$ satisfying the following three conditions.
\begin{itemize}
    \item[(a)] For any $x,y\in X$ such that $d(x,y)<2r,$ there is a unique minimizing geodesic $\gamma_{x,y}$ joining $x$ to $y.$ 
    \item[(b)] If $x,y,a\in M$ such that  $d(x,y)< 2 r,$ $d(a,x)<r,$ $d(a,y)<r$  and $z$ is a point on the shortest geodesic joining $x$ to $y,$ then $d(a,z)\leq \max(d(a,x),d(a,y)).$
    \item[(c)] if $\gamma$ and $\gamma'$ are minimizing geodesics and if $0\leq s,s'< 2r$ and $0\leq t\leq 1,$ then $d(\gamma(st),\gamma'(s't))\leq d(\gamma(s),\gamma'(s')).$ 
\end{itemize}

Assume that $0<r<r(X).$  Then for any $\alpha=(x_0,\dots,x_n)\in (\VR^\infty_{<r}X)_n,$ we consider a map  (see \cite[p.179]{hausmann1994vietoris}) 
\begin{equation}
T(\alpha) : \Delta^n_{\sf top} \longrightarrow X
\end{equation}
that can be inductively defined as follows. Let $(e_i)_{i\in [n]}$ be the standard basis in $\RR^{[n]}.$ Set $T(\alpha)(e_0)=x_0.$ Suppose $T(\alpha)(y)$ is defined for elements of the form  $y=\sum_{i=0}^{p-1} t_i e_i$ for some $p\leq n$ and define it for an element of the from $z=\sum_{i=0}^{p} t_i e_i.$  If $t_p=1,$ we set $T(\alpha)(e_p)=x_p.$ Otherwise, we consider a point $x=T^\alpha(\frac{1}{1-t_p} \sum_{i=0}^{p-1} t_i e_i)$ and set $T(\alpha)(z)=\gamma_{x,x_p}(t_p\cdot d(x,x_p))$. It is easy to see that the diagrams 
\begin{equation}
\begin{tikzcd}[column sep = 1pt]
\Delta^{n-1}_{\sf top} \ar{rr}{\partial^i}\ar{rd}[below left]{T(\partial_i \alpha)} & &  \Delta^n_{\sf top} \ar{ld}{T(\alpha)} \\ 
& X &
\end{tikzcd}
\hspace{1cm}
\begin{tikzcd}[column sep = 1pt]
\Delta^{n+1}_{\sf top} \ar{rr}{s^i}\ar{rd}[below left]{T(s_i \alpha)} & &  \Delta^n_{\sf top} \ar{ld}{T(\alpha)} \\ 
& X &
\end{tikzcd}
\end{equation}
are commutative. Therefore, this construction defines a map $T:|\VR^\infty_{<r}X|\to X$ that can be restricted to a map
\begin{equation}\label{eq:T}
T^\nu:|\VR^\nu_{<r}X|\longrightarrow X.
\end{equation}

\begin{theorem}[{cf. \cite[(3.5)]{hausmann1994vietoris}}]\label{th:manifold} Let $X$ be a geodesic metric space with $r(X)>0.$ Then for any $0<r<r(X)$ and any distance matrix norm $\nu,$ the map \eqref{eq:T} is a homotopy equivalence. 
\end{theorem}
\begin{proof} 
For any $x\in X,$ we set $U_x=B_r(x),$ where $B_r(x)=\{y\in X\mid d(x,y)<r\}$ is an open ball, and consider the family $\UU=(U_x)_{x\in X}.$ For any set of points $\sigma=\{x_0,\dots,x_n\},$ we set $U_\sigma=U_{x_0}\cap \dots \cap U_{x_n}.$ We claim that these intersections $U_\sigma$ are closed with respect to geodesics: 
\begin{equation}
x,y\in U_\sigma \ \Rightarrow \ \gamma_{x,y}(t)\in  U_\sigma 
\end{equation}
for all $t.$ Indeed, by property (b) we obtain 
\begin{equation}
d(x_i,\gamma_{x,y}(t))\leq \max(d(x_i,x),d(x_i,y)) <r.   
\end{equation}

Let us prove that $U_\sigma$ is shortly contractible. Choose a point $u\in U_\sigma.$ Consider a map 
\begin{equation}
h : U_\sigma \times [0,2r] \longrightarrow U_\sigma, \hspace{1cm} h(x,t) = \gamma_{u,x}\left( \frac{td(u,x)}{2r} \right).  
\end{equation}
Note that $h(x,2r)=x$ and $h(x,0)=u.$ By condition (c) we obtain 
\begin{equation}
d(h(x,t),h(x',t)) \leq d(x,x'). 
\end{equation}
Since $\gamma_{x,y}$ is an arc-length parameterized shortest geodesic, we have 
\begin{equation}
d(h(x',t),h(x',t')) \leq \frac{d(u,x')}{2 r} \cdot |t-t'|  \leq |t-t'|.
\end{equation}
Therefore, we have 
\begin{equation}
d(h(x,t),h(x',t')) \leq d(x,x') + |t-t'|  
\end{equation}
for any $x,x'\in U_\sigma$ and $t\in [0,2r].$ 
It follows that $h$ is a short map and $U_\sigma$ is shortly contractible. 

For $x\in X,$ we consider a simplicial subset $U'_x = \VR^\nu_{<r} U_x\subseteq \VR^\nu_{<r} X.$ The inclusion \eqref{eq:elli:B} combined with Lemma \ref{lemma:union:family} imply that $\UU'=(U'_x)_{x\in X}$ is a cover of $\VR^\nu_{<r} X.$ By Corollary \ref{cor:diag:cover} we obtain that the inclusions $U'_\sigma \hookrightarrow \VR^\nu_{<r} X$ induce a weak equivalence from the diagonal of the \v{C}ech complex  $\diag \Ch(\UU')\to \VR^\nu_{<r} X.$  Then \cite[Th. 3.6]{isaacsonexercises} implies that the inclusions induce a weak equivalence from the homotopy colimit 
\begin{equation}
\hocolim_{\Delta^{op}} \Ch(\UU') \overset{\sim} \longrightarrow \VR^\nu_{<r} X.
\end{equation}
On the other hand, \cite[Th.1.1]{dugger2004topological} implies that the inclusions $U_\sigma\hookrightarrow X$ induce a weak equivalence 
\begin{equation}
\hocolim_{\Delta^{op}} \Ch(\UU) \overset{\sim} \longrightarrow X.
\end{equation}
Since $U_\sigma$ is closed with respect to geodesics, the map $T^\nu$ can be restricted to maps   $T^\nu_\sigma : |U'_\sigma| \to U_\sigma.$  Using that  $U_\sigma$ is shortly contractible, we obtain that  $|U'_\sigma|$ and $U_\sigma$ are contractible  (Corollary  \ref{cor:short_contr}), and hence, $T^\nu_\sigma$ is a homotopy equivalence. Therefore the statement follows from the fact that geometric realisation commutes with homotopy colimits \cite[Th.18.9.10]{hirschhorn2003model} and the commutativity of the following diagram
\begin{equation}
\begin{tikzcd}
\hocolim_{\Delta^{op}} |\Ch(\UU')| \ar{d}{\sim} \ar{r}[below]{\sim}{(T^\nu_\sigma)} & \hocolim_{\Delta^{op}} \Ch(\UU) \ar{d}{\sim} \\
 {|\VR^\nu_{<r} X |} \ar{r}{T^\nu} & X
\end{tikzcd}
\end{equation}
\end{proof}

It is well known that a compact Riemannian manifold $M$ has $r(M)>0$ \cite[p.179, Remark 1]{hausmann1994vietoris}.

\begin{corollary}
For a compact Riemannian manifold $M,$ there is $r(M)>0$ such that for any $0<r<r(M)$ and any distance matrix norm $\nu$ there is a homotopy equivalence 
\begin{equation}
|\VR^\nu_{<r} M| \simeq M,
\end{equation}
and for $n\geq 1$ we have an isomorphism 
\begin{equation}
\MH_{n,<r}(M)\cong H_n(M).
\end{equation}
\end{corollary}

\section{Example: circle}
Given a metric space $X$ and a distance matrix norm $ \nu $, we define $r_\nu(X)$ as the infimum of positive real numbers  $r$ such that the space $|\VR^\nu_{<r} X|$ is not homotopy equivalent to $X$  
\begin{equation}
   r_\nu(X)=\inf\{r>0\mid |\VR^{\nu}_{<r} X|\not\simeq X\}.   
\end{equation}
We also set 
\begin{equation}
r_p(X) = r_{\nu_p}(X), \hspace{1cm} r^{{\sf sym}}_p(X) = r_{\nu^{\sf sym}_p}(X).
\end{equation}
For a geodesic metric space $X$, Theorem \ref{th:manifold} implies that $r_\nu(X)\geq r(X)$. 

We define the circle as
\begin{equation}
  S^1=\RR/\ZZ.  
\end{equation}
The counterclockwise distance on $S^1$ is defined as $\vec{d}(x,y)=\tilde y-\tilde x,$ where $\tilde x,\tilde y\in \RR$ are liftings of $x$ and $y$ respectively such that $\tilde x\leq \tilde y < \tilde x+1.$ The distance on $S^1$ is defined as 
\begin{equation}
 d(x,y)=\min \{\vec{d}(x,y), \vec{d}(y,x)\}.   
\end{equation}
This section is devoted to the proof of the following theorem.
\begin{theorem}\label{th:circle}
    For any $p\in [1,\infty],$ we have 
	\[
	    r_p(S^1)=r_p^{\sf sym}(S^1)=\frac{2^{\frac{1}{p}}}{2+2^{\frac{1}{p}}}.
	\]
\end{theorem}

Further we denote by $\nu$ any distance matrix norm and, for any metric space $X$, we set:
\begin{equation}
    \begin{aligned}
        C^{\nu}_{n,<r}(X)&:=C_n(\VR^{\nu}_{<r}(X);\ZZ/2)=\ZZ/2\cdot[(\VR^{\nu}_{<r} X)_n];\\
	Z^{\nu}_{n,<r}(X)&:={\rm Ker}\: \delta_n;\\
	B^{\nu}_{n,<r}(X)&:=\mathrm{Im}\: \delta_{n+1};\\
	H^{\nu}_{n,<r}(X)&:=Z^{\nu}_{n,<r}(X)/B^{\nu}_{n,<r}(X),
    \end{aligned}
\end{equation}
where $\delta_n:C^\nu_{n,<r}(X)\to C^\nu_{n-1,<r}(X)$ is the $n$th boundary map. For any $a,b\in C^\nu_{1,<r}(X),$ we write 
$a \equiv b$
if $a-b\in B^\nu_{1,<r}(X).$

\begin{lemma}\label{lem:order_change}
For any $(x,y)\in C^\nu_{1,<r}(S^1),$  we have 
$(x,y) \equiv (y,x).$
\end{lemma}
\begin{proof}
First note that there exists $a\in S^1$ such that $x,y\in [a,a+1/2].$ The interval $[a,a+1/2]\subseteq S^1$ is isometric to the interval $[\tilde a,\tilde a+1/2]\subseteq \RR,$ where $\tilde a$ is a lift of $a.$ Then Corollary  \ref{cor:short_contr} implies that $Z^\nu_{1,<r}([a,a+1/2])=B^\nu_{1,<r}([a,a+1/2]).$ The statement follows. 
\end{proof} 

\begin{definition}
We say that a sequence of distinct points $x_0,\dots,x_{n-1}\in S^1$ indexed by the elements of the cyclic group $\ZZ/n$ is \emph{cyclically ordered}, if 
\begin{equation}
\sum_{i\in \ZZ/n} \vec{d}(x_i,x_{i+1}) = 1.
\end{equation}
Any finite subset $F\subset S^1$ can be cyclically ordered i.e. presented as $F=\{x_0,\dots,x_{n-1}\},$ where $x_0,\dots,x_{n-1}$ is cyclically ordered. We say that a cyclically ordered sequence $x_0,\dots,x_{n-1}$ is \emph{$r$-dense}, if $\vec{d}(x_i,x_{i+1})<r$ for any $i\in \ZZ/n.$ 
A $1$-cycle $c\in C^\nu_{1,<r}(S^1)$ is \emph{$r$-dense  cyclically ordered}, if it can be presented as 
\begin{equation}
c= \sum_{i\in \ZZ/n}(x_i,x_{i+1}),
\end{equation}
where $x_0,\dots,x_{n-1}$ is an $r$-dense cyclically ordered sequence. 
\end{definition}

\begin{lemma}\label{lem:cycle:directed_cycle} For any $r>0$, either $H^\nu_{1,<r}(S^1)=0$ or $H^\nu_{1,<r}(S^1)\cong \ZZ/2.$ Moreover, if $H^\nu_{1,<r}(S^1)\cong \ZZ/2,$ then any $r$-dense cyclically ordered $1$-cycle represents the non-trivial element. 
\end{lemma}
\begin{proof}
First assume that $0<r<1/4=r(S^1).$ 
Then Theorem \ref{th:manifold} implies that the map $T^\nu : |\VR^\nu_{<r} S^1|\to S^1$ is a homotopy equivalence. 
Therefore in the case $0<r<1/4$ the statement follows from the fact that $T^\nu$ induces an isomorphism for homology $H^\nu_{1,<r}(S^1) \cong H_1(S^1;\ZZ/2)$ and the description of homology of the circle $H_1(S^1;\ZZ/2)\cong \ZZ/2.$ 

Now assume that $r>0$ is arbitrary and take $s=\min(r,1/8).$ 
We claim that the map $H^\nu_{1,<s}(S^1)\to H^\nu_{1,<r}(S^1)$ is an epimorphism. 
Take a cycle $c=\sum_{i}(x_i,y_i)\in Z^\nu_{1,<r}(S^1).$ By Lemma \ref{lem:order_change}, replacing $c$ by a homologous cycle, we can assume that $d(x_i,y_i)=\vec{d}(x_i,y_i).$ Define the subdivision ${\sf sd}(c)$ of $c$ in the following way. 
Consider a point $z_i$ in ``between'' of $x_i$ and $y_i$ such that $\vec{d}(x_i,z_i)=\vec{d}(z_i,y_i)=\vec{d}(x_i,y_i)/2.$ It is easy to see that $d(x_i,z_i)=\vec{d}(x_i,z_i)$ and $d(z_i,y_i)=\vec{d}(z_i,y_i).$ Then set ${\sf sd}(c)=\sum_i (x_i,z_i)+ \sum_i (z_i,y_i).$ Since $\delta(x_i,z_i,y_i)=(z_i,y_i)+(x_i,y_i)+(x_i,z_i)$ and $\WW_\nu(x_i,z_i,y_i)\leq d(x_i,z_i)+d(z_i,y_i)=d(x_i,y_i)<r$ we obtain that $c$ and ${\sf sd}(c)$ represent the same element in $H^\nu_{1,<r}(S^1).$ On the other hand, it is easy to see that ${\sf sd}^n(c)\in Z^\nu_{1,<s}(S^1)$ for sufficiently big  $n.$ Hence, the map $H^\nu_{1,<s}(S^1)\to H^\nu_{1,<r}(S^1)$ is an epimorphism. 
Moreover, if $c$ is $r$-dense  cyclically ordered, then for sufficiently big $n$ the cycle ${\sf sd}^n(c)$ is $s$-dense cyclically ordered. The statement follows. 
\end{proof}

\begin{lemma}\label{lem:semicircle}
A finite subset of $S^1$ is contained in an open semicircle if and only if any its subset of cardinality at most 3 is contained in an open semicircle.
\end{lemma}
\begin{proof} First we note that if $F\subset S^1$ is a finite subset and $F=\{x_0,\dots,x_{n-1}\}$ is a cyclic order, then  $F$ in included in an open semicircle if and only if $\vec{d}(x_i,x_{i+1})>\frac{1}{2}$ for some $i.$ Let us prove the statement by induction on the cardinality of the subset $F$, which we denote by $n$. If $n\leq 3,$ the statement is obvious. Assume that $n\geq 4$ and the statement holds for all subsets of cardinality less than $n,$ and prove it for sets of cardinality $n.$ Assume the contrary, that there is a subset $F\subset S^1$ of cardinality $n$ such that any its subset of cardinality at most $3$ lies in an open semicircle, but $F$ does not lie in an open semicircle. The assumption implies that any proper subset of $F$ lies in an open semicircle.  The fact that $F$ is not contained in an open semicircle means that $\vec{d}(x_i,x_{i+1})\leq \frac{1}{2}$ for any $i\in \ZZ/n.$ Since $\{x_0,\dots,x_{i-1},x_{i+1},\dots,x_n\}$ lies in an open semicircle, we obtain $\vec{d}(x_{i-1},x_{i+1}) > \frac{1}{2}.$ Therefore 
\begin{equation}
\frac{n}{2} < \sum_{i\in \ZZ/n} \vec{d}(x_{i-1},x_{i+1}) = \sum_{i\in \ZZ/n} (\vec{d}(x_{i-1},x_{i}) + \vec{d}(x_i,x_{i+1})) = 2,  
\end{equation}
which makes the contradiction. 
\end{proof}

\begin{definition} 
We denote by $t(\nu)$ the infimum of the values of $\WW_\nu(x_0,x_1,x_2),$ where the triple $x_0,x_1,x_2\in S^1$ is not contained in any open semicircle.
\end{definition}

\begin{proposition}\label{prop:homology_of_circle:null} For any $\nu,$ we have 
$r_\nu(S^1) = t(\nu).$  Moreover, if $r\leq  t(\nu),$ then $|\VR^\nu_{<r} S^1|\simeq S^1$; and if $r>t(\nu)$ then  $ H^\nu_{1,<r}(S^1)=0.$ 
\end{proposition}
\begin{proof}
Assume that $r\leq t(\nu).$ We denote by $U_x = B_{1/4}(x)$ the open semicircle with the center in $x$ and  consider the cover of the circle by all open semicircles $\UU=(U_x)_{x\in S^1}.$  By Theorem \ref{theorem:nerve_theorem}, in order to prove that  $|\VR^\nu_{<r} S^1|\simeq S^1,$ it is sufficient to prove that 
\begin{equation}
\VR^\nu_{<r} S^1 = 
\bigcup_{x\in S^1} \VR^{\nu}_{<r} U_x. 
\end{equation}
Therefore we need to prove that, for any $(x_0,\dots,x_n) \in (\VR^\nu_{<r} S^1)_n,$ the set  $\{x_0,\dots,x_n\} $ is included in an open semicircle. For $n\leq 3,$ it follows from the definition of $t(\nu).$ Then for any $n$ it follows from Lemma \ref{lem:semicircle}.

Assume that $r> t(\nu).$ Then  there exists a triple  $(x_0,x_1,x_2)\in C^\nu_{2,<r}(S^1)$ such that $\{x_0,x_1,x_2\}$ is not contained in any open semicircle. Hence 
    \begin{equation}
      (x_1,x_2)+(x_0,x_2)+(x_0,x_1) = \delta(x_0,x_1,x_2) \equiv 0.  
    \end{equation}
    Lemma \ref{lem:order_change} implies that
    \begin{equation}
      (x_{\pi(0)},x_{\pi(1)})+(x_{\pi(1)},x_{\pi(2)})+(x_{\pi(2)},x_{\pi(0)})\equiv   0
    \end{equation}
    for any permutation $\pi.$ Take $\pi$ such that $x_{\pi(0)},x_{\pi(1)},x_{\pi(2)}$ is cyclically ordered. Since these three points do not lie in any open semicircle, we obtain that 
    \begin{equation}
     \vec{d}(x_{\pi(i)},x_{\pi(i+1)})=d(x_{\pi(i)},x_{\pi(i+1)})   
    \end{equation}
    for $i\in \ZZ/3.$ Hence $\vec{d}(x_{\pi(i)},x_{\pi(i+1)})< r.$
    Therefore $x_{\pi(0)},x_{\pi(1)},x_{\pi(2)}$ is an $r$-dense  cyclically ordered sequence. It follows that there exists an $r$-dense  cyclically ordered 1-cycle, which is null homologous. Then Lemma \ref{lem:cycle:directed_cycle} implies that $H^\nu_{1,<r}(S^1)=0.$
\end{proof}

\begin{lemma}\label{lem:l_p:circle}
$t(\nu_p)=t(\nu_p^{\sf sym})=\frac{2^{\frac{1}{p}}}{2+2^{\frac{1}{p}}}$.
\end{lemma}
\begin{proof}
Suppose that $\{x_0,x_1,x_2\}$ is not contained in any semicircle, then we must have $d(x_0,x_1)+ d(x_1,x_2) + d(x_2,x_0)=1$. Set $a=d(x_0,x_1)$ and $b=d(x_1,x_2)$. By H\"{o}lder's inequality, we have 
    \begin{equation}
        \frac{2}{2^{\frac{1}{p}}}\|a,b\|_p=\|a,b\|_p\cdot\|1,1\|_q\geq a+b,
    \end{equation}
    where $\frac{1}{p} +  \frac{1}{q}=1$. Since $\WW_p(x_0,x_1,x_2)=\max\{\|a,b\|_p,1-a-b\}$, we have 
    \begin{equation}
       \left(1+ \frac{2}{2^{\frac{1}{p}}} \right) \WW_p(x_0,x_1,x_2) \geq (1-a-b) + \frac{2}{2^{\frac{1}{p}}} \| a,b\|_p \geq 1.
    \end{equation}
Hence we obtain  $\WW_p(x_0,x_1,x_2)\geq \frac{2^{\frac{1}{p}}}{2+2^{\frac{1}{p}}}.$ It also follows that $\WW^{\sf sym}_p(x_0,x_1,x_2)\geq \frac{2^{\frac{1}{p}}}{2+2^{\frac{1}{p}}}.$ Hence 
 $t(\nu_p) \geq  \frac{2^{\frac{1}{p}}}{2+2^{\frac{1}{p}}}$ and $t(\nu^{\sf sym}_p) \geq  \frac{2^{\frac{1}{p}}}{2+2^{\frac{1}{p}}}.$ 
    
    On the other hand, we can construct the following tuple
    \begin{equation}
        (x_0,x_1,x_2)=(\frac{-1}{2+2^{\frac{1}{p}}} ,0,\frac{1}{2+2^{\frac{1}{p}}}).
    \end{equation}
    Since 
    \begin{equation}
    d(x_0,x_2) = 1 - \frac{2}{2+2^{\frac{1}{p}}} = \frac{2^{\frac{1}{p}}}{2+ 2^{ \frac{1}{p}}} = \| d(x_0,x_1),d(x_1,x_2) \|_p,
    \end{equation}
    we obtain 
    \begin{equation}
\WW_p(x_0,x_1,x_2) = \WW_p^{\sf sym}(x_0,x_1,x_2) = \frac{2^{\frac{1}{p}}}{2+2^{\frac{1}{p}}}.
    \end{equation}
Using that  $\vec{d}(x_i,x_{i+1})\leq \frac{1}{2},$ we see that they are not contained in any open semicircle. Hence $t(\nu_p)\leq \frac{2^{\frac{1}{p}}}{2+2^{\frac{1}{p}}}$ and $t(\nu^{\sf sym}_p)\leq \frac{2^{\frac{1}{p}}}{2+2^{\frac{1}{p}}}.$
\end{proof}

\begin{proof}[Proof of Theorem \ref{th:circle}]
It follows from Proposition \ref{prop:homology_of_circle:null} and Lemma \ref{lem:l_p:circle}. 
\end{proof}

\section{Pro-simplicial sets and limit homology} 

In this section we show that the limits of homology of the simplicial set $\VR^\nu_{<r} X,$ as $r$ tends to zero, do not depend on $\nu.$ Moreover, we will show that pro-simplicial sets defined by  $\VR^{\nu,n}_{<r} X$ do not depend on $\nu.$ First, we will remind some basic information about pro-objects. 

\subsection{Reminder about pro-objects} 

Let us recall the definition of the category of pro-objects $\Pro(\CC)$ in a category $\CC$ \cite{kashiwara2006indization}, \cite[\S 1.1]{mardesic1982shape}, 
\cite{isaksen2002calculating}. Consider the co-Yoneda embedding 
\begin{equation}
y^{op} : \CC \hookrightarrow \Fun(\CC,\Set)^{op}, \hspace{1cm} y^{op}(c) = \CC(c,-).
\end{equation}
For any small cofiltered category $I$ and a functor $\FF:I\to \CC,$ we define the pro-object of $\FF$  as an object of  $\Fun(\CC,\Set)^{op}$ given by 
\begin{equation}
\FF_{\sf pro} = \lim \ y^{op} \FF. 
\end{equation}
Sometimes $\FF_{\sf pro}$ is also called formal cofiltered limit of $\FF.$
Then the category of pro-objects $\Pro(\CC)$ is the full subcategory of $\Fun(\CC,\Set)^{op}$ consisting of pro-objects $\FF_{\sf pro},$ where $\FF$ runs over all functors from small cofiltered categories. The category of ind-objects ${\sf Ind}(\CC)$ is defined dually ${\sf Ind}(\CC)=\Pro(\CC^{op})^{op}.$ The hom-set between the pro-objects of functors $\FF:I\to \CC$ and $\GG:J\to \CC$ can be computed as follows (see \cite[p. 134]{kashiwara2006indization})
\begin{equation}
\Hom_{\Pro(\CC)}(\FF_{\sf pro},\GG_{\sf pro}) \cong \underset{J}\lim\: \underset{I^{op}}\colim\: \Hom_\CC(\FF i,\GG j). 
\end{equation}

Let us give a more explicit description of the hom-set. First assume that $\GG$ is a constant functor  sending all objects to a fixed object $c$ and all morphisms to ${\sf id}_c$ and describe the hom-set $\Hom_{\Pro(\CC)}(\FF_{\sf pro},c)$ (here we identify $c$ with $y^{op}(c)$). We say that two morphisms $\alpha_1:\FF i_1\to c$ and $\alpha_2:\FF i_2\to c$ are equivalent $\alpha_1\sim \alpha_2$, if there is an object $i_3$ together with morphisms $f_1:i_3\to i_1$ and $f_2:i_3\to i_2$ such that $\alpha_1 (\FF f_1) =\alpha_2 (\FF f_2).$ Using that $I$ is a cofiltered category, it is easy to check that this is an equivalence relation. An equivalence class of a morphism $\alpha:\FF i\to c$ will be denoted by $[\alpha].$  Then
\begin{equation}
\Hom_{\Pro(\CC)}(\FF_{\sf pro},c)\cong  \underset{I^{op}}\colim\: \Hom_\CC(\FF i,c) 
\end{equation}
can be described as the set of the equivalence classes $[\alpha:\FF i \to c].$ Now assume that $\GG$ is an arbitrary functor. Therefore 
\begin{equation}
\Hom_{\Pro(\CC)}(\FF_{\sf pro},\GG_{\sf pro})\cong \underset{J}\lim\ \Hom_{\Pro(\CC)}(\FF_{\sf pro},\GG j). 
\end{equation}
It follows that a morphism $\varphi:\FF_{\sf pro}\to \GG_{\sf pro}$ is a family of equivalence classes $\varphi=([\varphi_j : \FF i_j\to \GG j])_j$ such that for any morphism $g:j\to j'$ we have $[(\GG g) \varphi_j] = [\varphi_{j'}].$ For any tree functors from small cofiltered categories $\FF:I\to \CC$, $\GG:J\to \CC$ and $\HH:K\to \CC$ and morphisms $\varphi:\FF\to \GG$ and $\psi:\GG\to \HH$ the composition $\psi\varphi$ is defined component-wise $\psi\varphi=([ \psi_k \varphi_{j_k}])_k.$

The construction $\Pro(-)$ is natural with respect to all functors \cite[Prop.6.1.9]{kashiwara2006indization}. More precisely, for any functor $\Phi:\CC\to \DD,$ there exists a unique functor $\Pro(\Phi) : \Pro(\CC)\to \Pro(\DD)$ that commutes with cofiltered limits such that the diagram 
\begin{equation}
\begin{tikzcd}
\CC \ar[rr,"\Phi"] \ar[d] && \DD \ar[d] \\
\Pro(\CC) \ar[rr,"\Pro(\Phi)"] && \Pro(\DD)
\end{tikzcd}
\end{equation}
is commutative. In particular, we obtain that 
\begin{equation}
\Pro(\Phi)(\FF_{\sf pro})\cong  (\Phi\FF)_{\sf pro}.
\end{equation}
In other words, any functor commutes with formal small cofiltered limits. 

If $\CC$ admits all small cofiltered limits, then the embedding $\CC \hookrightarrow \Pro(\CC)$ has a left adjoint sending $\FF_{\sf pro}$ to $\lim \FF$ \cite[Prop 6.3.1]{kashiwara2006indization}
\begin{equation}
L : \Pro(\CC) \longrightarrow \CC,
\hspace{1cm}
L(\FF_{\sf pro}) \cong \lim \FF.
\end{equation}
Using what was said above, we obtain that, if two functors from small cofiltered categories $\FF: I \to \CC$ and $\GG: J\to \DD$ have isomorphic pro-objects, then for any functor to a category admitting cofiltered limits 
$\Phi:\CC\to \DD$ the limits of compositions $\Phi \FF$ and $\Phi\GG$ are isomorphic
\begin{equation}\label{eq:pro->lim}
\FF_{\sf pro}\cong \GG_{\sf pro} \hspace{5mm} \Rightarrow \hspace{5mm} \lim \Phi\FF \cong \lim \Phi \GG.
\end{equation}

\subsection{Vietoris-Rips pro-simplicial sets and limit homology}

Let $X$ be a metric space, $\nu$ be a distance matrix norm and $n\geq 0$. We can consider the $n$-skeleton of the $\nu$-Vietoris-Rips simplicial as a functor from the cofiltered category of positive real numbers 
\begin{equation}
\VR^{\nu,n}_{<*} X : (0,\infty) \longrightarrow \sSet. 
\end{equation}
This functor defines a pro-simplicial set that we denote by 
\begin{equation}
\VR^{\nu,n}_{\sf pro} X \in  \Pro(\sSet).
\end{equation}
The following lemma states that strict and non-strict versions of the simplicial sets define the same pro-simplicial sets. 

\begin{lemma}
The non-strict version of the functor
\begin{equation}
\VR^{\nu,n}_{\leq *} X : (0,\infty)\to \sSet 
\end{equation}
defines a pro-simplicial set isomorphic to $\VR^{\nu,n}_{\sf pro} X.$ The isomorphism is induced by the embedding $\VR^{\nu,n}_{<r} X \hookrightarrow \VR^{\nu,n}_{\leq r} X$. 
\end{lemma}
\begin{proof} Denote by $\VR^{\nu,n}_{\leq \sf pro} X$ the pro-object defined by $\VR^{\nu,n}_{\leq *} X.$ Note that there are inclusions 
\begin{equation}\label{eq:inclusions:e/2}
\VR^{\nu,n}_{<r/2} X \subseteq \VR^{\nu,n}_{\leq r/2} X \subseteq \VR^{\nu,n}_{<r} X \subseteq \VR^{\nu,n}_{\leq r} X.
\end{equation} 
The inclusions $\varphi_{r}:\VR^{\nu,n}_{<r} X \hookrightarrow \VR^{\nu,n}_{\leq r} X$ and $\psi_r:\VR^{\nu,n}_{\leq r/2} X\to \VR^{\nu,n}_{<r} X$ define morphisms of pro-simplicial sets $\varphi=([\varphi_r])_r: \VR^{\nu,n}_{\sf pro} X \to \VR^{\nu,n}_{\leq \sf pro} X$ and $\psi=([\psi_r])_r:\VR^{\nu,n}_{\leq \sf pro} X \to \VR^{\nu,n}_{\sf pro}.$ Using the inclusions \eqref{eq:inclusions:e/2}, it is easy to check that their compositions $\varphi\psi$  and $\psi \varphi$ are identity morphisms. 
\end{proof}

\begin{theorem}\label{th:pro-simplicial}
For a metric space $X,$ a distance matrix norm $\nu$ and $n\geq 0,$ the inclusions $\VR^{\nu,n}_{<r} X \hookrightarrow  \VR^{\infty,n}_{<r} X$ define an isomorphism of pro-simplicial sets
\begin{equation}
\VR^{\nu,n}_{\sf pro} X \cong \VR^{\infty,n}_{\sf pro} X.
\end{equation}
\end{theorem}
\begin{proof}
By \eqref{eq:inclusion-to-infty}  we have inclusions 
\begin{equation}\label{eq:inclusions:nu-infty}
 \VR^{\nu,n}_{< r/C_n}  X \subseteq  \VR^{\infty,n}_{< r/C_n}  X \subseteq  \VR^{\nu,n}_{<r} X \subseteq \VR^{\infty,n}_{<r} X,
\end{equation}
where $C_n=C_n(\nu).$ Then the inclusions $\varphi_r: \VR^{\nu,n}_{<r} X  \to  \VR^{\infty,n}_{<r} X$ and $\psi_r : \VR^{\infty,n}_{< r/C_n}  X \to  \VR^{\nu,n}_{<r} X$ define morphisms of pro-simplicial sets $\varphi = ([\varphi_r])_r: \VR^{\nu,n}_{\sf pro} X  \to  \VR^{\infty,n}_{\sf pro} X$ and $\psi=([\psi_r])_r : \VR^{\infty,n}_{\sf pro}  X \to  \VR^{\nu,n}_{\sf pro} X.$ Using the inclusions \eqref{eq:inclusions:nu-infty}, it is easy to check that $\varphi\psi={\sf id}$ and $\psi\varphi={\sf id}.$
\end{proof}

Assume that $\KK$ is a commutative ring and consider the functor 
\begin{equation}
H_n(\VR^{\nu}_{<*} X) : (0,\infty) \longrightarrow {\sf Mod}(\KK). 
\end{equation}
where $H_n(-)=H_n(-;\KK)$ is the homology with coefficients in $\KK.$ The limit of this functor is denoted by 
\begin{equation}
\LH^\nu_n(X) := \lim_{r}\  H_n(\VR^\nu_{<r} X).
\end{equation}
If $\nu=\nu_\infty,$ the limit is the limit of homology of the classical Vietoris-Rips complex
\begin{equation}
\LH_n(X):=\LH^{\nu_\infty}_n(X) = \lim_{r} \ H_n( \vr^\infty_{<r} X )
\end{equation}
considered by Vietoris in the pioneering article  \cite{vietoris1927hoheren}. Since the pro-simplicial sets $\VR^{\nu,n+1}_{<*}X$ and $\VR^{\nu,n+1}_{\leq *}X$ are isomorphic, we obtain that the pro-modules  $H_n(\VR^{\nu,n+1}_{<*}X)$ and $H_n(\VR^{\nu,n+1}_{\leq *}X)$ are also isomorphic, and hence 
\begin{equation}
\LH_n^{\nu}(X) \cong  \lim_{r} H_n(\VR^\nu_{\leq r} X). 
\end{equation}
As a corollary of Theorem \ref{th:pro-simplicial}, we obtain the following generalization of the statement Nina Otter \cite[Th.33]{otter2018magnitude}.
\begin{corollary}
For any metric space $X$ and any distance matrix norm $\nu,$ the inclusion $\VR^\nu_{<r}X \to \VR^\infty_{<r}X$ induces an isomorphism 
\begin{equation}
\LH^\nu_*(X) \cong \LH_*(X).
\end{equation}
\end{corollary}
\begin{proof} It follows from 
Theorem \ref{th:pro-simplicial} and \eqref{eq:pro->lim}. 
\end{proof}

\section{Filtered colimits of metric spaces}\label{sec:filtered_colimits}

The content of this section is the result of a personal discussion with Emily Roff, to whom we express our gratitude. Here we prove that $\VR^\nu_{<r}$ commutes (up to homotopy) with filtered colimits of metric spaces and short maps. As a corollary, we obtain that strictly blurred magnitude homology commutes with filtered colimits.

A disadvantage of the ordinary magnitude homology theory is that it does not commute with filtered colimits of metric spaces in the category of metric spaces and short maps. Indeed, if we consider two-element metric spaces $X_n=\{x_n,y_n\},$  where the distance is defined by  $d(x_n,y_n)=1+\frac{1}{n},$ and take the obvious short maps $X_n\to X_{n+1},$ then $\colim X_n=X=\{x,y\},$ where $d(x,y)=1.$ However, $\MH_{1,1}(X_n)=0$ for $n<\infty,$ and $\MH_{1,1}(X)=\langle (x,y),(y,x) \rangle  \cong  \ZZ^2.$ Therefore 
\begin{equation}
 \colim  \MH_{1,1}(X_n) \ncong \MH_{1,1}(X). 
\end{equation}
It is also easy to see that $\MH_{1,\leq 1}(X_n)=0$ and $\MH_{1,\leq 1}(X)=\langle (x,y)+(y,x) \rangle \cong \ZZ.$ Therefore 
\begin{equation}
 \colim  \MH_{1,\leq 1}(X_n) \ncong \MH_{1,\leq 1}(X). 
\end{equation}
However, further we prove that strictly blurred magnitude homology commute with filtered colimits.

In this section it will be convenient to slightly extend the category of metric spaces to the category of ``pseudometric'' spaces. By a pseudometric space we mean a set $X$ equipped by a map $d:X\times X\to \RR_{\geq 0}$ such that $d(x,x)=0,$   $d(x,y)=d(x,y)$ and $d(x,y)+d(y,z) \geq d(x,z)$ for all $x,y,z\in X.$ So the distance between different points can be equal to zero. Short maps between pseudometric spaces is defined in the same way as for metric spaces. The category of pseudometric spaces and short maps is denoted by $\PMet.$ Then the category of metric spaces and short maps $\Met$ is its full subcategory
\begin{equation}
\Met \subset \PMet.
\end{equation}
The Vietoris-Rips complex $\VR^\nu_{<r} X$ of a pseudometric space $X$ is defined in the way similar to the case of metric spaces. 

For any pseudometric space $X,$ there is a universal map to a metric space called Kolmogorov quotient of $X$ 
\begin{equation}
X \longrightarrow \Kol(X).
\end{equation}
The metric space $\Kol(X)$ is defined as the quotient $X/\sim$ by the equivalence relation $\sim$ such that $x\sim y$ if and only if $d(x,y)=0.$ The distance in $\Kol(X)$ is defined by $d([x],[y])=d(x,y).$ Therefore, the category of metric spaces is a reflective subcategory of the category of pseudometric spaces. 

Later we show that filtered colimits in $\PMet$ are defined in the straightforward way, but filtered colimits in $\Met$ are Kolmogorov quotients of the filtered colimits in the category of pseudometric spaces. We will need the following lemma to compare Vietoris-Rips simplicial sets of colimits in the category of metric spaces and in the category of pseudometric spaces. 

\begin{lemma}\label{lemma:triv_fib} 
Let $\nu$ be a distance matrix norm and $r>0.$ 
For any pseudometric space $X,$ the map $X\to \Kol(X)$ induces a trivial fibration of the $\nu$-Vietoris-Rips simplicial sets
\begin{equation}
\VR^\nu_{<r} X \overset{\sim}\epi \VR^\nu_{<r} (\Kol(X)).
\end{equation}
\end{lemma}
\begin{proof}
 By  \cite[Ch I, Th.11.2]{goerss2009simplicial}, we have to check that any lifting problem of the form 
\begin{equation}\label{eq:lifting_prop}
\begin{tikzcd}
\partial \Delta^n \ar[d,hookrightarrow] \ar[r,"\psi"] & \VR^\nu_{<r} X \ar[d] \\
 \Delta^n \ar[r,"\varphi"] \ar[ur,dashed] & \VR^\nu_{<r} (\Kol(X)) 
\end{tikzcd}
\end{equation}
has a solution. If $n=0,$ this follows from the fact that $X\to \Kol(X)$ is surjective. If $n=1,$ it follows from the fact that $\WW_\nu(x_0,x_1)= d(x_0,x_1)=d([x_0],[x_1])=\WW_\nu([x_0],[x_1]).$ Let us prove it for $n\geq 2.$ In general, since $\partial \Delta^n$ is the  coequaliser 
\begin{equation}
\begin{tikzcd}
\coprod\limits_{0\leq i<j\leq n} \Delta^{n-2} 
\ar[r,shift left=1mm]
\ar[r,shift left=-1mm]
& \coprod\limits_{0\leq i\leq n} \Delta^{n-1} \ar[r] & \partial \Delta^n
\end{tikzcd}
\end{equation}
(see \cite[p.9]{goerss2009simplicial}), for a simplicial set $A,$ a simplicial map $\partial \Delta^n\to A$ is uniquely defined by a sequence of simplices $a^{(0)},\dots,a^{(n)}\in A_{n-1}$ such that $d_i(a^{(j)})=d_{j-1}(a^{(i)})$ for $0\leq i<j\leq n.$ Therefore, a map $\psi : \partial \Delta^n \to \VR^\nu_{<r}  X$ is defined by a sequence of tuples $(x_{0}^{(i)},\dots,x_{i-1}^{(i)},x_{i+1}^{(i)},\dots,x_n^{(i)})\in (\VR^{\nu}_{<r} X)_{n-1}, 0\leq i\leq n$ such that
\begin{equation}
 (x_{0}^{(i)},\dots,\hat x_{i}^{(i)}, \dots, \hat x_{j}^{(i)},\dots  x_n^{(i)}) 
 =
 (x_{0}^{(j)},\dots,\hat x_{i}^{(j)}, \dots, \hat x_{j}^{(j)},\dots  x_n^{(j)}),
\end{equation}
for any $0\leq i<j\leq n,$ where the hat $\hat{(-)}$ means that we skip the entry. Since $n\geq 2,$ it follows that $x_k^{(i)}=x_k^{(j)}$ for any $k\notin \{i,j\}.$ Therefore, if we set $x_k=x_k^{(i)}, i\neq k,$  we obtain a tuple $(x_0,\dots, x_n).$ The map $\varphi:\Delta^n\to \VR^\nu_{<r} (\Kol(X))$ is defined by a tuple $(y_0,\dots,y_n)\in (\VR^\nu_{<r} (\Kol(X)))_n.$ The commutativity of the diagram \eqref{eq:lifting_prop} implies that 
\begin{equation}
([x_0],\dots,\widehat{[x_i]},\dots,[x_n]) = (y_0,\dots,\hat y_i,\dots,y_n)
\end{equation}
for any $0\leq i\leq n.$ Hence $[x_k]=y_k$ for any $0\leq k\leq n.$ Therefore, $([x_0],\dots,[x_n])\in (\VR^\nu_{<r} (\Kol(X)))_n.$ Since  $\WW_\nu(x_0,\dots,x_n)=\WW_\nu([x_0],\dots,[x_n]),$ we obtain $(x_0,\dots,x_n)\in  (\VR^\nu_{<r}  X)_n.$ Therefore the tuple $(x_0,\dots,x_n)$ defines a solution of the lifting problem $\Delta^n \to \VR^\nu_{<r}  X$. 
\end{proof}

Denote by $U:\PMet \to \Set$ the forgetful functor to the category of sets. If $\FF:I\to \PMet$ is a functor from a small filtered category, we take the filtered colimit in the category of sets $X=\colim U\FF$ with the universal cone $\alpha_i:\FF(i)\to X$ and endow it by a pseudometric defined by 
\begin{equation}
 d(x,y) = \inf\{d(x',y')\mid x',y'\in \FF(i), \alpha_i(x')=x, \alpha_i(y')=y \}.
\end{equation}
Since $I$ is filtered, for any finite subset $F\subseteq X,$ there is $i\in I$ and $F'\subseteq \FF(i)$ such that $\alpha_i(F')=F.$ Using this we obtain that the distance is well defined and that the triangle inequality is satisfied. So $(X,d)$ is a pseudometric space. It is easy to check that it is the colimit in the category of pseudometric spaces $(X,d) = \colim \FF.$ In particular, we obtain 
\begin{equation}
U( \colim \FF ) \cong \colim U\FF.
\end{equation} 

\begin{lemma}\label{lemma:colim_pseudo}
Let $\nu$ be a distance matrix norm and $r>0.$ Then for any small filtered category $I$ and any functor $\FF:I\to \PMet,$ we have an isomorphism 
\begin{equation}
\colim (\VR^\nu_{<r} \FF(i)) \cong  \VR^\nu_{<r} (\colim \FF).   
\end{equation}
\end{lemma}
\begin{proof} Set $X=\colim \FF.$ 
Since $(\VR^\nu_{<r}  \FF(i))_n \subseteq \FF(i)^{n+1}$ and $(\VR^\nu_{<r} X)_n \subseteq  X^{n+1},$ it is sufficient to prove that $(\VR^\nu_{<r} X)_n = \bigcup_i \alpha_i^{n+1}( \VR^\nu_{<r}  \FF(i) )_n.$ Let us prove this. Take a tuple $(x_0,\dots,x_n)\in (\VR^\nu_{<r}  X)_n$ and choose $r>0$ such that $\WW_\nu(x_0,\dots,x_n) + \varepsilon C_n(\nu)  < r.$
Since $I$ is filtered, using the definition of the distance on $X,$ we obtain that there exists $i$ and a tuple $(x'_0,\dots,x'_n)\in \FF(i)^{n+1}$ such that $\alpha_i(x'_k)=x_k$ and   $d(x'_k,x'_j) \leq d(x_k,x_j)+\varepsilon$ for any $0\leq k,j\leq n.$ Then by Lemma \ref{lemma:main_inequality} we obtain that $\WW_\nu(x_0',\dots,x_n') \leq \WW_\nu(x_0,\dots,x_n) + \varepsilon C_n(\nu) < r.$ Therefore, $(x_0',\dots,x_n')\in (\VR^\nu_{<r} \FF(i))_n$ is the preimage of $(x_0,\dots,x_n).$
\end{proof}

The category of metric spaces is not closed with respect to filtered colimits in the category of pseudometric spaces. For example, if $X_n=\{x_n,y_n\}$ is a metric space, such that $d(x_n,y_n)=\frac{1}{n},$ then the filtered colimit in the category of pseudometric spaces consists of two points $\{x_\infty,y_\infty\}$ such that $d(x_\infty,y_\infty)=0.$ However, if we denote by $\iota:\Met\to \PMet$ the inclusion functor, then for any functor from a filtered category $\FF:I\to \Met$ we have 
\begin{equation}\label{eq:colim_met}
\colim \FF \cong \Kol(\colim \iota \FF).
\end{equation}
So the filtered colimit in the category of metric spaces is the Kolmogorov quotient of the filtered colimit in the category of pseudometric spaces.

\begin{theorem}\label{th:filtered_colimits} Let $\nu$ be a distance matrix norm and $r>0.$  Then for any functor from a small filtered category to the category of metric spaces and short maps $\FF:I\to \Met,$ the simplicial map
\begin{equation}
\colim (\VR^\nu_{<r} \FF(i)) \overset{\sim}\longrightarrow  \VR^\nu_{<r} (\colim \FF) 
\end{equation}
is a weak equivalence.
\end{theorem}
\begin{proof}
This follows from Lemma \ref{lemma:colim_pseudo}, isomorphism  \eqref{eq:colim_met}, and Lemma \ref{lemma:triv_fib}.
\end{proof}

\begin{corollary}
For any distance matrix norm $\nu,$ any $r>0$ and any $n\geq 0,$ 
the functor $H_n(\VR^\nu_{<r} -)$ commutes with filtered colimits of metric spaces and short maps. 
\end{corollary}
\begin{proof}
This follows from the fact that the homology of spaces commutes with filtered colimits. 
\end{proof}

\begin{corollary}
Strictly blurred magnitude homology commutes with filtered colimits of metric spaces and short maps. 
\end{corollary}

\section{Appendix. Simplicial sets vs.  simplicial complexes} \label{sec:ss_sc}

 We show that there is an adjunction between the category of simplicial sets and simplicial complexes, which defines an equivalence between the category of simplicial complexes and some full subcategory of simplicial sets. We also show that this equivalence agrees with the geometric realization up to homotopy. The results of this section were inspired by the post of John Baez \enquote{Simplicial Sets vs. Simplicial Complexes} in \enquote{The $n$-category Caf\'{e}\:}, its comments and the discussion in the proof of \cite[Prop. 9]{otter2018magnitude} in the paper of Nina Otter. These results are well known to experts, but we decided to add them here for completeness.

By a simplicial complex $K,$ we mean a set of finite non-empty sets such that $\emptyset \neq \tau \subseteq \sigma \in K$ implies $\tau \in K.$ The set of vertices of $K$ is defined by $V(K)=\bigcup_{\sigma\in K} \sigma.$ A morphism of simplicial complexes $f:K\to L$ is a map $f:V(K)\to V(L)$ such that for any $\sigma \in K$ we have $f(\sigma)\in L.$ The category of simplicial complexes is denoted by $\SCpx.$ 

For a simplicial complex $K,$ we denote by $\SS(K)$ a simplicial set, whose components are defined by
\begin{equation}
\SS(K)_n = \{ (x_0,\dots,x_n)\mid \{x_0,\dots,x_n\}\in K\},
\end{equation}
the $i$-th face map is defined by deletion of the $i$-th vertex, and the $i$-th degeneracy map is defined by doubling of the $i$-th vertex. 

For a simplicial set $X,$ we consider a \emph{map of the $i$-th vertex} 
\begin{equation}
v_i : X_n \longrightarrow X_0, \hspace{1cm} 0\leq i\leq n,
\end{equation}
which is induced by the map $[0]\to [n]$ sending $0$ to $i.$ For a simplex $x\in X_n,$ we set    
\begin{equation}
v_*(x) = (v_0(x),\dots,v_n(x)).
\end{equation}
It is easy to see that there are the following relations with face maps and degeneracy maps
\begin{equation}\label{eq:v:relations}
\begin{split} 
v_*(\partial_i x) &= ( v_0(x),\dots,v_{i-1}(x), v_{i+1}(x),\dots,v_n(x) ),
\\
v_*( s_i x ) &= (v_0(x),\dots,v_i(x),v_i(x),\dots,v_n(x)).
\end{split}
\end{equation}
Therefore we can consider a simplicial complex 
\begin{equation}
\SC(X) = \{ \{v_0(x),\dots,v_n(x)\} \mid n\geq 0, x\in X_n\}
\end{equation}
with a morphism of simplicial sets
\begin{equation}
v_* : X \longrightarrow \SS(\SC(X)).
\end{equation}
It is easy to see that for any simplicial complex $K$, there is an isomorphism 
\begin{equation}
\SC(\SS(K)) \cong K.
\end{equation}

These two constructions define functors between the category of simplicial sets and the category of simplicial complexes
\begin{equation}\label{eq:functorsSSSC}
\SC : \sSet \leftrightarrows \SCpx : \SS.
\end{equation}

\begin{proposition}\label{prop:sc_ss_adjunction}
The functors \eqref{eq:functorsSSSC} are adjoint 
\begin{equation}\label{eq:adjunction}
\Hom_\SCpx(\SC(X),K) \cong \Hom_\sSet(X,\SS(K)),  
\end{equation}
where $v_*:X\to \SS(\SC(X))$ is the unit of the adjunction, and the isomorphism $\SC(\SS(K))\cong K$  is the counit of the adjunction. 
\end{proposition}
\begin{proof} Let  $X$ be a simplicial set and $K$ be a simplicial complex. Then a morphism $f:\SC(X) \to K$ is a map of sets of  vertices $f: X_0 \to V(K)$ such that $f(\{v_0(x),\dots,v_n(x)\})\in K$ for any $x\in X_n$. Then the relations \eqref{eq:v:relations} imply that we can construct a morphism of simplicial sets $\tilde f:X\to \SS(K)$ defined by $\tilde f(x)=(f(v_0(x)),\dots,f(v_n(x))).$ We need to show that the correspondence $f\mapsto \tilde f$ defines a bijection \eqref{eq:adjunction}. 
The fact that the map is injective is obvious. 

Let us prove that it is surjective. Assume that we have a morphism of simplicial sets $g:X\to \SS(K).$ For $x\in X_0$ we set $f(x)=g(x)\in V(K).$ This defines a map $f:X_0\to V(K).$  For $x\in X_n$ we have $f(v_i(x))=v_i(g(x)).$ Therefore
\begin{equation}
g(x)=(f(v_0(x)),\dots, f(v_n(x))).
\end{equation}
It follows that $\{f(v_0(x)),\dots, f(v_n(x))\}\in K.$ Hence $f$ defines  a morphism of simplicial complexes $f:\SC(X)\to K$ such that $g=\tilde f.$
\end{proof}

\begin{corollary}
The category of simplicial complexes $\SCpx$ is equivalent to the full subcategory of the category of simplicial sets $\sSet$ consisting of simplicial sets $X$ such that the map $v_*:X \to \SS(\SC(X))$ is an isomorphism. 
\end{corollary}

Further we show that this equivalence of categories agrees with geometric realizations up to homotopy (Theorem \ref{th:geometric:realisation:SSK}). 

For a set $S,$ we denote by $\Delta(S)$ the simplicial complex consisting of all non-empty finite subsets of $S.$ 

\begin{lemma}\label{lemma:contractible:simplicial:set}
For a non-empty set $S,$ the simplicial set $\SS(\Delta(S))$ is contractible. 
\end{lemma}
\begin{proof}
Consider the category $\tilde S,$ whose objects are elements of $S$ and each hom set consists of one morphism. Then $\SS(\Delta(S))$ is isomorphic to the nerve of this category $N(\tilde S).$ Moreover, the category $\tilde S$ is equivalent to the one-point category $*.$ Therefore $\SS(\Delta(S)) \simeq N(\tilde S) \simeq N(*)=*.$  
\end{proof}

For a simplicial complex $K,$ we denote by $P(K)$ the poset of its simplices.

\begin{lemma}\label{lemma:colimit:SS}
For a simplicial complex $K,$ the inclusions $\Delta(s) \to K, s\in K,$ induce an isomorphism 
\begin{equation}
\underset{s\in P(K)}\colim  \SS(\Delta(s)) \cong \SS(K).
\end{equation}
\end{lemma}
\begin{proof} Let us first formulate a general statement about colimits of sets. 
Let $S$ be a set and  $F:\mathcal{C}\to {\sf Sets}$ be a functor from a small category to the category of sets such that: (1) $F(c)\subseteq S;$ (2) for any morphism $f:c\to c'$ the map $F(f):F(c)\to F(c')$ is the inclusion $F(c)\subseteq F(c')$; (3) for any element from an intersection $s\in F(c) \cap F(c')$ there is an object $c''$ with  morphisms $f:c''\to c$ and $g:c''\to c'$ such that $s\in F(c'').$ Then $\colim F=\bigcup_{c\in C}F(c).$ Indeed, any natural transformation to a constant functor 
$\varphi : F\to \Delta T$ uniquely defines a  well defined map $\tilde \varphi:\bigcup_{c\in C} F(c) \to T$ such that $\tilde \varphi(s)=\varphi_c(s)$ for $s\in F(c).$ Since colimits of simplicial sets are defined component-wise, we can apply this statement to simplicial sets.

Therefore, the statement of the lemma follows from the following facts. For $s\in K$ the embedding $\Delta(s)\subseteq K$ induces an embedding $\SS(\Delta(s))\subseteq \SS(K).$ The simplicial set $\SS(K)$ can be presented as the union $\SS(K)=\bigcup_{s\in P(K)} \SS( \Delta(s) ).$ For any $s,t\in K$ such that $s\cap t\neq \emptyset$ we have $\SS(\Delta(s))\cap \SS(\Delta(t)) = \SS(\Delta(s\cap t))$ and $s\cap t\in K.$ 
\end{proof}

\begin{theorem}\label{th:geometric:realisation:SSK}
For a simplicial complex $K,$ there is a natural homotopy equivalence of geometric realizations 
\begin{equation}
|K| \simeq |\SS(K)|.
\end{equation}
\end{theorem}
\begin{proof}
Let us remind the notion of a free diagram of simplicial sets \cite[\S 2.4]{dwyer1983function}, \cite[\S 2.4]{farjoun2006homotopy}, \cite[Appendix I]{farjoun1996cellular}. A functor $F:\mathcal{C}\to \sSet$ is called free, if it is possible to choose a collection of subsets (called basis of the functor)  $B_{n,c}\subseteq F(c)_n$ for $n\geq 0, c\in \mathcal{C},$ which are closed with respect to degeneracy maps $\sigma_i(B_{n,c} )\subseteq B_{n+1,c},$ and such that for any simplex $x\in F(c)_n$ there exists a unique $f:c'\to c$ and a unique element of the basis $b\in B_{n,c'}$ such that $x=F(f)(b).$ For a free functor $F$ the comparison map $\hocolim F \to \colim F$ is a weak equivalence \cite[\S 4.2]{dwyer1983function}.

We claim that the functor $P(K)\to \sSet, s \mapsto \SS(\Delta(s))$ is free. Indeed the sets 
\begin{equation}
B_{n,s}=\{(x_0,\dots,x_n)\in \SS(\Delta(s)) \mid \{x_0,\dots,x_n\} = s \}    
\end{equation}
form its basis. Combining this with Lemma \ref{lemma:colimit:SS}  we obtain that there is a natural weak equivalence
\begin{equation} \underset{s\in P(K)}\hocolim\    \SS(\Delta(s)) \overset{\sim}\longrightarrow \SS(K).
\end{equation}
On the other hand, it is well known that for any category $C$  the homotopy colimit of the trivial functor is the nerve $\hocolim_{C} *\simeq N(C).$ Therefore, using Lemma \ref{lemma:contractible:simplicial:set},  we obtain a weak equivalence 
\begin{equation}
\underset{s\in P(K)}\hocolim\   \SS(\Delta(s)) \overset{\sim}\longrightarrow N(P(K)).
\end{equation}
Applying the geometric realisation, we obtain that there is a natural homotopy equivalence
\begin{equation}
|N(P(K))| \simeq |\SS(K)|.
\end{equation}
Then the result follows from the homotopy equivalences 
\begin{equation}
|N(P(K))| \simeq |{\sf sd}(K)| \simeq  |K|,
\end{equation}
where ${\sf sd}(K)$ is the barycentric subdivision of $K.$ 
\end{proof}

\section{Appendix. Nerve theorem for simplicial sets}

This section is devoted to the nerve theorem for simplicial sets (Theorem \ref{th:simplicial_nerve_theorem}). The results of this section are well known to experts, but we decided to add them here for completeness.

\begin{definition}[\v{C}ech nerve]
Let $\mathcal{C}$ be a small category with pullbacks and  let $ f:X\to Y$ be a morphism in $ \mathcal{C}$. Its corresponding \emph{\v{C}ech nerve} $\Ch(f) $ is the simplicial object in $ \mathcal{C}$ whose $n$-th component is given by the $(n+1)$-fold fiber product of $X$ over $Y$: 
\begin{equation}
\Ch(f)= \left(
\begin{tikzcd}
\dots 
\ar[shift left = 1.5mm]{r} 
    \ar[shift left = 0.5mm]{r}  
    \ar[shift left = -0.5mm]{r}
    \ar[shift left = -1.5mm]{r}
& 
X \times_Y X \times_Y X 
\ar[shift left=1mm]{r}
\ar{r} 
\ar[shift left=-1mm]{r}
&
    X\times_Y X 
    \ar[shift left = 0.5mm]{r}  
    \ar[shift left = -0.5mm]{r}
    & X 
\end{tikzcd}
\right).
\end{equation}
The degeneracy maps are defined as projections and degeneracy maps are defined as diagonal embeddings. Note that  the map $f$ induces a morphism of simplicial objects
\begin{equation}
\hat f : \Ch(f) \longrightarrow Y,
\end{equation}
where $Y$ is treated as a constant simplicial object.
\end{definition}

\begin{lemma}\label{lemma:surjective:Chech}
For a surjective map of sets $f:X\to Y,$ the morphism of simplicial sets $\hat f : \Ch(f)\to Y$ is a weak equivalence to a constant simplicial set. 
\end{lemma}
\begin{proof}
First assume that $Y=\{y\}$ is a one-point set. Then $\Ch(f)\cong \SS(\Delta(X))$ is contractible by Lemma \ref{lemma:contractible:simplicial:set}. Now consider the general case of a surjective map $f:X\to Y.$ Since $\Ch(f)_n \cong \coprod_{y\in Y} f^{-1}(y)^{n+1},$ the map to the discrete simplicial set  $\hat f:\Ch(f)\to Y$ is a disjoint union of maps $\hat f_y : \Ch(\hat f_y) \to \{y\}$ for $y\in Y,$ where $f_y: f^{-1}(y)\to \{y\}$ is the restriction of $f.$ Using that each of the maps $\hat f_y$ is a homotopy equivalence, we obtain that $\hat f$ is a homotopy equivalence.
\end{proof}

If $\mathcal{C}={\sf sSet}$ is the category of simplicial sets, then $\Ch(f)$ is a bisimplicial set, where $\Ch(f)_{n,m}=X_n^{\times_{Y_n}(m+1)}$ is the fiber product of $m+1$ copies of $X_n$ over $Y_n.$

\begin{proposition}\label{prop:chech:diagonal}
Let $f : X \to Y$ be a component-wise surjective morphism of simplicial sets. Then $\hat f$ induces a weak homotopy equivalence
\begin{equation}
    \diag (\Ch(f)) \overset{\sim}\longrightarrow Y.
\end{equation}
\end{proposition}
\begin{proof} By Lemma \ref{lemma:surjective:Chech}, for any $n,$ the map 
 $(\hat{f}_n) =  \hat{f}_{n, \bullet}:\Ch(f)_{n,\bullet}\to Y_{n}$ is a homotopy equivalence, where $Y_n$ is treated as a constant simplicial set. Then the result follows from  \cite[Prop.1.9]{goerss2009simplicial}. 
\end{proof}

Let $X$ be a simplicial set and $\UU=(U_i)_{i\in I}$ be a family of its simplicial subsets. For a subset $\sigma \subseteq I,$ we set $U_\sigma = \bigcap_{i\in \sigma} U_i.$ We denote by $N_\UU$ the nerve complex of this family 
\begin{equation}
N_\UU = \{\sigma \subseteq I\mid U_\sigma \neq \emptyset,\: \sigma \text{ is non-empty finite}\}.
\end{equation}
The poset of simplices of $N_\UU$ is denoted by $P_\UU.$ 

\begin{definition}[\v{C}ech nerve of a cover] Let $X$ be a simplicial set and $\UU=(U_i)_{i\in I}$ be its cover i.e. a family of simplicial subsets such that $X=\bigcup_{i\in I}U_i$. 
If we take $Y=\coprod_{i\in I}U_i,$ then the inclusions $U_i\to X$ induce a component-wise surjective simplicial map $f:Y\to X.$ We set $\Ch(\UU):=\Ch(f).$ Since    $U_{i_0}\times_X  \dots \times_X U_{i_m}  \cong U_{i_0} \cap \dots \cap U_{i_m},$ we obtain 
\begin{equation}
\Ch(\UU)_{\bullet,m} \cong \coprod_{\substack{(i_0,\dots,i_m)\in I^{m+1} : \\ \{ i_0,\dots,i_m \}\in N_\UU}} U_{i_0} \cap \dots \cap U_{i_m}.  
\end{equation}
\end{definition}

\begin{corollary}\label{cor:diag:cover}
For a cover $\UU=(U_i)_{i\in I}$ of a simplicial set $X,$ the inclusions $U_{i_0} \cap \dots \cap U_{i_m} \hookrightarrow X$ induce a weak equivalence  
\begin{equation}
\diag \Ch(\UU) \overset{\sim}\longrightarrow X.
\end{equation}
\end{corollary}

\begin{theorem}[Nerve theorem for simplicial sets]\label{th:simplicial_nerve_theorem}
Let $X$ be a simplicial set and $\UU=(U_i)_{i\in I}$ be a family of its simplicial subsets such that $X=\bigcup_{i\in I} U_i$ and  $U_\sigma$ is weakly contractible for any $\sigma\in N_\UU.$ Then there is a homotopy equivalence
\begin{equation}
|X| \simeq |N_\UU|. 
\end{equation}
\end{theorem}
\begin{proof}
Note that for the simplicial set associated with the simplicial complex $N_\UU,$ there is a formula
\begin{equation}
\SS(N_\UU)_m \cong \coprod_{\substack{(i_0,\dots,i_m)\in I^{m+1} : \\ \{ i_0,\dots,i_m \}\in N_\UU}} *.   
\end{equation}
Therefore, if we denote by $S$ a bisimplicial set such that $S_{n,m}=\SS(N_\UU)_m,$ we obtain a morphism of bisimplicial sets $\Ch(\UU) \to S,$ sending $U_{i_0} \cap \dots \cap U_{i_m}$ to $*.$ Since $\Ch(\UU)_{\bullet,m} \to S_m$ is a weak equivalence for any $m$, by \cite[Prop.1.9]{goerss2009simplicial}, we obtain that the morphism $\diag( \Ch(\UU) ) \to \SS(N_\UU)$ is a weak equivalence. Therefore the result follows from Proposition \ref{prop:chech:diagonal} and Proposition \ref{th:geometric:realisation:SSK}. 
\end{proof}

\begin{proposition}
Let  $\varphi:X\to Y$ be a morphism of simplicial sets and  $\UU=(U_i)_{i\in I}$ and $\VV=(V_i)_{i\in I}$ are covers of $X$ and $Y$ respectively such that $\varphi(U_i)\subseteq V_i.$ Assume that for any $i_0,\dots,i_n\in I$ the map 
\begin{equation}
U_{i_0}\cap \dots \cap U_{i_n} \overset{\sim}\longrightarrow V_{i_0} \cap \dots \cap V_{i_n} 
\end{equation}
is a weak equivalence. Then $\varphi$ is a weak equivalence. 
\end{proposition}
\begin{proof} The result 
\cite[Prop.1.9]{goerss2009simplicial} implies that  $\varphi$ induces a weak equivalence $\diag \Ch(\UU) \overset{\sim}\to \diag \Ch(\VV).$ Then the result follows from Proposition \ref{prop:chech:diagonal} and commutativity of the diagram
\begin{equation}
\begin{tikzcd}
\Ch(\UU) \ar{r}{\sim} \ar{d}{\sim} & \Ch(\VV) \ar{d}{\sim} \\
X \ar{r} & Y
\end{tikzcd}
\end{equation}
\end{proof}

\end{document}